\pdfoutput=1
\documentclass[reqno]{amsart}
\usepackage[foot]{amsaddr}
\PassOptionsToPackage{isbn=false,url=false,bibencoding=utf8,maxbibnames=5}{biblatex}
\PassOptionsToPackage{hidelinks}{hyperref}
\PassOptionsToPackage{capitalize,nameinlink,noabbrev}{cleveref}
\usepackage[margin=1.5in]{geometry}
\input{definitions}

\title{Partial Algebraic Shifting}
\author[Della Vecchia]{Antony Della Vecchia\textsuperscript{1}}
\author[Joswig]{Michael Joswig\textsuperscript{1,2}}
\author[Lenzen]{Fabian Lenzen\textsuperscript{1}}
\address{%
	\textsuperscript{1}Technische Universität Berlin, Chair of Discrete Mathematics/Geometry\newline
	\textsuperscript{2}Max-Planck Institute for Mathematics in the Sciences, Leipzig
}
\thanks{\textit{E-mail addresses.} \texttt{(\href{vecchia@math.tu-berlin.de}{vecchia}|\href{mailto:joswig@math.tu-berlin.de}{joswig}|\href{mailto:lenzen@math.tu-berlin.de}{lenzen})@math.tu-berlin.de}}

\begin{document}
	\begin{abstract}
		We study algebraic shifting of uniform hypergraphs and finite simplicial complexes in the exterior algebra with respect to matrices which are not necessarily generic.
Several questions raised by Kalai (2002) are addressed.
For instance, it turns out that the combinatorial shifting of Erd\H{o}s--Ko--Rado (1961) arises as a special case.
Moreover, we identify a sufficient condition for partial shifting to preserve the Betti numbers of a simplicial complex; examples show that this condition is sharp.
	\end{abstract}
	\maketitle\thispagestyle{empty}

\section{Introduction}
Algebraic shifting is a widely applicable method for converting a uniform hypergraph $S$ into another hypergraph, $\Delta(S)$, which is somehow simpler but still retains key properties of $S$.
The algebraic contribution comes from a generic linear transformation of a certain vector space, suitably crafted from $S$.
This procedure can be applied to an abstract simplicial complex, by treating each layer of $(k{-}1)$-faces as a $k$-uniform hypergraph.
In this way, the (exterior or symmetric) algebraic shift $\Delta(K)$ of a simplicial complex $K$ is a simplicial complex again, with the same f-vector and the same Betti numbers as $K$.
Algebraic shifting was introduced by Kalai~\cite{Kalai84,Kalai02} and further studied by Björner--Kalai~\cite{BjornerKalai:1988} and many others.
Here we restrict our attention to algebraic shifting in the exterior algebra.
It had been observed that it formally makes sense to replace the aforementioned generic linear transformation (in the exterior algebra) by an arbitrary one.
In fact,~\cite[Problem~10]{Kalai02} asks to describe how such shifting with respect to a nongeneric transformation (which we call \enquote{partial shifting})
differs from shifting with respect to a generic transformation (which we call \enquote{full shifting}).
The purpose of this article is to explore the situation, for arbitrary ground fields, and give some answers.
To the best of our knowledge, this is the first attempt in this direction.

One original motivation for studying algebraic shifting is the relationship to combinatorial shifting introduced by Erd\H{o}s--Ko--Rado~\cite{ErdosKoEtAl:1961}.
In that landmark article in extremal combinatorics the authors prove an explicit upper bound for the number of pairwise overlapping hyperedges in uniform hypergraphs; see also Frankl~\cite{Frankl:1987} for a survey.
We identify combinatorial shifting as a special case of partial shifting.
In this way, our present work offers a unified approach to (exterior) algebraic shifting and combinatorial shifting.
For instance, this will allow us to identify where these two concepts agree and how they differ.

A second train of thought is concerned with viewing combinatorial and algebraic shifting through the lens of commutative algebra and algebraic geometry.
Herzog--Tarai~\cite{HerzogTerai:1999} and Aramova--Herzog~\cite{Aramova+Herzog:2000} related (full) algebraic shifting to generic initial ideals.
Combinatorial shifting occurs in work of Knutson~\cite{Knutson:1408.1261} on Schubert calculus.
Recently, Woodroofe \cite{Woodroofe:2022} gave a new proof of Erd\H{o}s--Ko--Rado via a fixed point result on algebraic groups; see also Bulavka--Gandini--Woodroofe \cite{BulavkaGandiniWoodroofe:2406.17857} for extensions.

Our main results are the following.
First, we observe a fundamental connection between partial exterior shifting of a uniform hypergraph $S$ on $n$ points and the Bruhat decomposition of the general linear group $\GL(n,\EE)$, with respect to its standard $(B,N)$-pair, over a suitably chosen field $\EE$; for $(B,N)$-pairs see Abramenko--Brown~\cite[Chapter 6]{Abramenko+Brown:2008}.
As a consequence, we obtain a special case of partial shifting parameterized by permutations in the symmetric group $\SymmetricGroup{n}$, which is the Weyl group of $\GL(n,\EE)$.
Recall that $\GL(n,\EE)$ always admits a Bruhat decomposition over any field $\EE$, even if $\EE$ is not perfect; see, e.g., Bump's monograph on Lie groups~\cite[\S30]{Bump:2013}.
Second, this description allows us to recognize the full shift as the partial shift with respect to the longest element in $\SymmetricGroup{n}$.
In this language, combinatorial shifting with respect to a simple transposition can be realized as partial exterior shifting with respect to the same transposition.
This observation provides one possible answer to a question asked in~\cite[Problem~40]{Kalai02}.
Third, partial shifting with respect to a permutation allows to vary the genericity of the shifting operation in a well-behaved way.
The extreme cases are given by combinatorial shifting and full shifting.
To analyze the situation further, we study the \emph{partial shift graph} of a uniform hypergraph $S$; the latter is a directed graph which represents all partial shifts of $S$ with respect to all permutations.
In particular, we show that the partial shift graph is acyclic; this addresses~\cite[Problem~7]{Kalai02}.
Fourth, equipped with all of the above we prove that the Betti numbers of a simplicial complex are preserved by partial shifts with respect to permutations which are large enough in the weak order.
The latter result constitutes a vast generalization of known properties of algebraic shifting and a partial answer to Kalai's question concerning the topological ramifications of nongeneric shifting~\cite[Problem~10]{Kalai02}.
Our implementation of partial shifting in the computer algebra system \OSCAR~\cite{OSCAR-book} helped us finding interesting examples.
The algorithmic details will be discussed elsewhere.

This article is organized as follows.
Section \ref{sec:shifting} serves the purpose of fixing our notation; we usually consider $k$-uniform hypergraphs on $n$ vertices.
The Bruhat decomposition of $\GL(n,\EE)$ is the topic of \cref{sec:bruhat}.
It is not difficult to see that the partial shift only depends on the (right) coset of the Borel group of upper triangular matrices (\cref{thm:invariance upper triangular}).
Such cosets are known to be parameterized by permutations in the symmetric group $\SymmetricGroup{n}$.
Our main contribution here is an explicit method of systematically picking suitable matrix representatives for a given permutation (Proposition \ref{thm:partial shift by u(w)}).
On the way we see that the full shift amounts to partial shifting with respect to the longest element in $\SymmetricGroup{n}$ (Proposition \ref{prop:w0}).
Combinatorial shifting and its interpretation as partial shifting with respect to a transposition is addressed in \cref{sec:combinatorial}.
Section \ref{sec:partial-shift-graph} is dedicated to the partial shift graph $\ShiftGraph(n,k,m)$, whose nodes are the $k$-uniform hypergraphs on $n$ vertices with exactly $m$ edges.
Its directed edges correspond to partial shifts.
In \cref{thm:partial shift graph acyclic} we show that $\ShiftGraph(n,k,m)$ is acyclic.
Up to this point the focus of our work is on hypergraphs and their partial shifts.
In \cref{sec:simplicial} we change our focus by looking at simplicial complexes.
First of all, in this scenario there are many natural and interesting examples, e.g., the triangulation of the real projective plane $\RR\PP^2_6$ on six vertices (\cref {exmp:rp26:contracted}).
Moreover, only in this setup it makes sense to study any topological implications of partial shifting.
We prove that the partial shift with respect to a permutation $w \in \SymmetricGroup{n}$ preserves the Betti numbers if $w$ is at least as large in the weak order as the standard $n$-cycle $(1 \ 2 \ 3 \ \ldots \ n)$;
see \cref{thm:n-cycle}.
This generalizes a key fact known about full shifting~\cite[Theorem~3.1]{BjornerKalai:1988}.
Further, for $\RR\PP^2_6$ this result is also tight.
That is, each partial shift (in characteristic zero) of $\RR\PP^2_6$ by a permutation not at least as large as $(1 \ 2 \ 3 \ \ldots \ n)$ changes the Betti numbers (\cref{exmp:tight}).
The final \cref{sec:outlook} raises some open questions.

\smallskip


\paragraph{Acknowledgement.}
We would like to thank Gil Kalai and Eran Nevo for inspiring discussions on algebraic shifting.
Further we are indebted to Russ Woodroofe for pointing out Knutson's paper \cite{Knutson:1408.1261}.
The work of all authors is funded by the Deutsche Forschungsgemeinschaft (DFG, German Research Foundation).
Specifically, ADV and MJ were supported by \enquote{MaRDI (Mathematical Research Data Initiative)} (NFDI 29/1, project ID 460135501);
MJ was supported by \enquote{Symbolic Tools in Mathematics and their Application} (TRR 195, project ID 286237555);
MJ and FL were supported by \enquote{The Berlin Mathematics Research Center MATH$^+$} (EXC-2046/1, project ID 390685689).

\section{Algebraic shifting}
\label{sec:shifting}
For $n \in \NN$, denote by $\binom{[n]}{k}$ the set of $k$-element subsets of $[n] \coloneqq \{1,2,\dots,n\}$.
We view nonempty subsets of $\binom{[n]}{k}$ as $k$-uniform hypergraphs on $n$-elements.
The total order on $[n]$ induces a partial order $\leq$ on $\binom{[n]}{k}$, called the \emph{domination order}, given by $\{a_1 < \dotsb < a_k\} \leq \{b_1 < \dotsb < b_k \}$ if $a_i \leq b_i$ for all $i$.
A hypergraph $S \subseteq \binom{[n]}{k}$ is called \emph{shifted} if it is an initial set with respect to this order;
i.e., if $\sigma \in S$ and $\rho \leq \sigma$, then also $\rho \in S$.
Explicitly, this means that $\sigma \setminus \{j\} \cup \{i\} \in S$ for every $\sigma \in S$ and $i \notin \sigma$, $j \in \sigma$ with $i \leq j$.

The following construction assigns to every $S \subseteq \binom{[n]}{k}$ a shifted hypergraph $\Delta(S)$ that shares several combinatorial properties with $S$.
Let $\EE \supseteq \FF$ be a field extension and $g\in\GL(n,\EE)$ be an invertible matrix.
The \emph{lexicographic order} on $\binom{[n]}{k}$, given by $S \lex\leq T$ if $\min S \symdiff T \in S$,
where $\symdiff$ denotes the symmetric difference, is a total order refining $\leq$.
It allows us to identify $\binom{[n]}{k}$ with $\bigl[\binom{n}{k}\bigr]$.
Then the $k$th \emph{compound matrix} $g^{\wedge k}$ of $g$ is the $\binom{n}{k} \times \binom{n}{k}$-matrix
with $g^{\wedge k}_{\sigma\tau} = \det (g_{ij})_{i \in \sigma, j \in \tau}$ for $\sigma, \tau \in \binom{[n]}{k}$.
We denote the row and column of $g^{\wedge k}$ corresponding to $\sigma$ and $\tau$ by $g^{\wedge k}_{\sigma*}$ and $g^{\wedge k}_{*\tau}$, respectively.
For $S \subseteq \binom{[n]}{k}$, we write $g^{\wedge S}$ for the $\abs{S} \times \binom{n}{k}$-submatrix of $g^{\wedge k}$ with rows indexed by $S$.

\begin{definition}\label{def:partial-g-shift}
	The \emph{partial (exterior) shift} of $S$ by the invertible matrix $g$ is
	\begin{equation}
		\label{eq:def shift}
		\Delta_g(S) \coloneqq \Set[\big]{ \sigma \in \tbinom{[n]}{k}; g^{\wedge S}_{*\sigma} \notin \Span_{\EE}[\big]{(g^{\wedge S}_{*\rho})_{\rho \lex< \sigma}} }.
	\end{equation}
	In other words, $\Delta_g(S)$ is given by those sets $\sigma \in \binom{[n]}{k}$ whose corresponding column $g^{\wedge S}_{*\sigma}$ of $g^{\wedge S}$
	does not lie in the linear span of columns $g^{\wedge S}_{*\rho}$ left of it.
\end{definition}

We call $g$ \emph{generic} if all entries of $g$ are algebraically independent over $\FF$.
In this case, $\Delta(S)\coloneqq \Delta_g(S)$ is shifted and does not depend on $g$; see~\cite[Theorem~2.1]{Kalai90} and \cref{sec:bruhat} below.
Algebraic shifting can be considered with respect to any fixed total order on $\binom{[n]}{k}$ refining $\leq$; throughout we stick to the lexicographic order $\lex\leq$.
\cref{def:partial-g-shift} directly lends itself to a naive algorithm for computing full or partial shifts.
Our first example illustrates the method.

\begin{example}\label{exmp:algo}
	We consider the case $n=4$ and $k=2$, where $\FF=\GF{2}$, and $\EE = \FF(x_{11},\dotsc,x_{44})$ is the field of rational functions in the 16 indeterminates $x_{11}$, $x_{12},\ldots$,$x_{44}$ with coefficients in $\FF$.
	The $4{\times 4}$-matrix $g = (x_{ij})_{1\leq i, j\leq 4}$
	is invertible and generic.
	Its second compound matrix $g^{\wedge 2}=(\det g_{\sigma\tau} )_{\sigma, \tau \in \binom{[4]}{2}}$ is a $6{\times}6$-matrix.
	For instance, for $\sigma=\{1,2\}$ and $\tau=\{2,3\}$ we have
	\[
	g^{\wedge 2}_{\sigma\tau} \ = \ \det \begin{psmallmatrix} x_{12} & x_{13} \\ x_{22} & x_{23} \end{psmallmatrix} \ = \ x_{12}x_{23} + x_{13}x_{22} \enspace.
	\]
	Abbreviating subsets $\{i_1,\dotsc,i_k\} \subseteq [n]$ by $i_1\dotsm i_k$,
	the lexicographic ordering on the $2$-subsets $\binom{[4]}{2}$ of $\{1,2,3,4\}$ reads $12 < 13 < 14 < 23 < 24 < 34$, and this dictates our ordering of the rows and columns of $g^{\wedge 2}$.
	We let $S=\{\sigma,\tau\}=\{12,23\}$, and obtain
	\[
	g^{\wedge S} \ = \ \begin{psmallmatrix} x_{11}x_{22} + x_{12}x_{21} & x_{11}x_{23} + x_{13}x_{21} & x_{11}x_{24} + x_{14}x_{21} &  x_{12}x_{23} + x_{22}x_{13} &  x_{12}x_{24}+ x_{14}x_{22} &  x_{13}x_{24} + x_{14}x_{23} \\  x_{21}x_{32} + x_{22}x_{31} & x_{21}x_{33} + x_{23}x_{31} & x_{21}x_{34} + x_{24}x_{31} & x_{22}x_{33} + x_{23}x_{32} & x_{22}x_{34} + x_{24}x_{32} & x_{23}x_{34} + x_{24}x_{33}
	\end{psmallmatrix} \enspace.
	\]
	This $2{\times}6$-matrix has the row echelon form
	\begin{equation}\label{eq:algo:row-echelon-form}
		\begin{psmallmatrix} x_{11}x_{22} + x_{12}x_{21} & x_{11}x_{23} + x_{13}x_{21} & x_{11}x_{24} + x_{14}x_{21} &  x_{12}x_{23} + x_{13}x_{22} &  x_{12}x_{24}+ x_{14}x_{22} &  x_{13}x_{24} + x_{14}x_{23} \\[3pt]
			0 & \substack{x_{11}x_{21}x_{22}x_{33} \\ + x_{11}x_{21}x_{23}x_{32} \\ + x_{12}x_{21}^2x_{33} \\+ x_{13}x_{21}^2x_{32} \\ + x_{12}x_{21}x_{23}x_{31} \\ + x_{13}x_{21}x_{22}x_{31}} & \substack{x_{11}x_{21}x_{22}x_{34} \\+ x_{11}x_{21}x_{24}x_{32} \\ + x_{12}x_{21}^2x_{34} \\+ x_{14}x_{21}^2x_{32} \\ + x_{12}x_{21}x_{24}x_{31} \\+ x_{14}x_{21}x_{22}x_{31}} & \substack{x_{11}x_{22}^2x_{33} \\+ x_{11}x_{22}x_{23}x_{32} \\+ x_{12}x_{21}x_{22}x_{33} \\+ x_{13}x_{21}x_{22}x_{32} \\+ x_{12}x_{22}x_{23}x_{31} \\+ x_{13}x_{22}^2x_{31}} & \substack{x_{11}x_{22}^2x_{34} \\+ x_{11}x_{22}x_{24}x_{32} \\+ x_{12}x_{21}x_{22}x_{34} \\+ x_{14}x_{21}x_{22}x_{32} \\+ x_{12}x_{22}x_{24}x_{31} \\+ x_{14}x_{22}^2x_{31}} & \substack{x_{11}x_{22}x_{23}x_{34} \\+ x_{11}x_{22}x_{24}x_{33} \\+ x_{12}x_{21}x_{23}x_{34} \\+ x_{12}x_{21}x_{24}x_{33} \\+ x_{13}x_{21}x_{24}x_{32} \\+ x_{14}x_{21}x_{23}x_{32} \\+ x_{13}x_{22}x_{24}x_{31} \\+ x_{14}x_{22}x_{23}x_{31}}
		\end{psmallmatrix} \enspace ,
	\end{equation}
	whose leading independent columns are indexed by $12$ and $13$.
	Consequently, we obtain the shifted hypergraph $\Delta(S) = \Delta_g(S) = \{12,13\}$.
\end{example}
The considerable growth of the algebraic expressions even in such a small example shows that this direct approach is only of very limited use.
It is an open question if there is a general polynomial time algorithm for computing $\Delta(S)$ from $S$; see~\cite[Problem~8]{Kalai02}.
Recently, Keehn and Nevo found such an efficient algorithm for triangulations of surfaces~\cite{KeehnNevo:2024}.
\begin{remark}[The role of the fields $\EE$ and $\FF$]\label{rem:monte-carlo}
	By~\cite[Theorem 2.1.6]{Kalai02} the full shift only depends on the characteristic of the field $\FF$.
	So it suffices to consider field extensions $\EE \supseteq \FF$ where $\FF$ is a prime field, and $\EE$ is the field of rational functions over $\FF$ with $n^2$ indeterminates, as in \cref{exmp:algo}.
	In characteristic zero it is enough to look at, e.g., $\RR\supseteq\QQ$ because $\RR$ contains sufficiently many transcendentals.
	The latter observation gives rise to a Monte--Carlo algorithm for computing $\Delta(S)$ from $S$ with high probability in characteristic zero; see~\cite[\S2.6]{Kalai02}.
        We usually omit the field from the notation $\Delta(S)$.
	Where the field is not clear from the context, we write $\Delta^\FF(S)$.
\end{remark}

The following result is known, at least for shifting with generic matrices~\cite[Theorem~3.1]{BjornerKalai:1988}.
We provide a short proof to stress that the genericity is not essential.
\begin{lemma}
	\label{thm:shift preserves cardinality}
	For every $S \subseteq \binom{[n]}{k}$ and $g \in \GL(n,\EE)$, we have $\abs{\Delta_g(S)} = \abs{S}$.
\end{lemma}
\begin{proof}
	For a matrix $m$, let $r(m)$ be the sequence with $r_0(m) = 0$ and $r(m)_k = \rk[(m_{*j})_{j \leq k}]$ for $k > 0$.
	Denoting the predecessor of a $\sigma \in \binom{[n]}{k}$ in the lexicographic order by $\sigma-1$, we obtain
	\begin{equation}
		\label{eq:rank-sequence-shift}
		\Delta_g(S) = \Set[\big]{\sigma; r(g^{\wedge S})_{\sigma - 1} < r(g^{\wedge S})_{\sigma}}.
	\end{equation}
	Of course, $r(m)_{j} - r(m)_{j-1} \leq 1$ for any matrix $m$;
	i.e., the sequence $r(m)$ only has steps of height one.
	On the other hand, $g$ is invertible, so by Sylvester's determinant identity (see \cref{thm:compound matrix} below),
	also $g^{\wedge k}$ is invertible.
	In particular, $g^{\wedge k}$ and thus also $g^{\wedge S}$ has linearly independent rows,
	so $\rk g^{\wedge S} = \abs{S}$.
	Now the elements of $\Delta_g(S)$ correspond to the steps in $r(g^{\wedge S})$,
	which are of height one and of which therefore there are precisely $\abs{S}$ many.
\end{proof}

We conclude this sections with some remarks giving various equivalent views on the definitions of $\Delta(S)$ and $\Delta_g(S)$.
\begin{remark}[Compound matrices and exterior powers; see {\cite[\S2.1]{Kalai02}}]
	\label{rmk:exterior powers}
	The definition of $\Delta_g(S)$ can also be viewed using the exterior exterior algebra $\bigwedge V$ of the vector space $V = \EE^n$.
	Namely, if $g$ is a matrix that represents (with respect to a fixed basis) an endomorphism of some vector space $V$,
	then the compound matrix $g^{\wedge k}$ represents the endomorphism $g^{\wedge k}$
	of the exterior power $V^{\wedge k}$ with respect to the induced basis of $V^{\wedge k}$.
	Specifically, let $e_i$ denote the $i$th basis vector of $V$,
	and for $\sigma = \{\sigma_1 < \dotsb < \sigma_k\}$ let $e_{\sigma} \coloneqq e_{\sigma_1} \wedge \dotsb \wedge e_{\sigma_k}$.
	Then
	\[
	g^{\wedge k}(e_\sigma) \ = \ \sum_{\rho} \det (g_{\rho\sigma}) e_\rho \enspace.
	\]
	The determinant is precisely the $(\rho, \sigma)$-entry of the compound matrix $g^{\wedge k}$.
	For $\sigma$ as above, let $g_{\sigma} = ge_{\sigma_1} \wedge \dotsb \wedge ge_{\sigma_k} = g^{\wedge k} e_\sigma$.
	For $S \subset \binom{[n]}{k}$, let $V(S) = V^{\wedge k}/\Span{e_\sigma; \sigma \notin S}$.
	For $v \in V^{\wedge k}$, let $\bar{v}$ be the image of $v$ in $V(S)$.
	Then
	\[
	\Delta_g(S) \ = \ \Set[\big]{\sigma \in \tbinom{[n]}{k};
		g_\sigma \notin \Span_{V(S)}{\bar{g}_\rho | \rho \lex< \sigma}} \enspace.
	\]
\end{remark}
\begin{remark}[Shifting via generic initial ideals; see {\cite{HerzogTerai:1999}}]
	\label{rmk:gin}
	For the generic exterior shift $\Delta(S)$, the above remark can be rephrased as follows.
	For $S \subseteq \smash{\binom{[n]}{k}}$, let $K(S)$ be the simplicial complex with faces $S \cup \bigcup_{m < k} \smash{\binom{[n]}{m}}$,
	and let $I_S \coloneqq I_{K(S)} \coloneqq (e_\sigma\colon \sigma \notin K)$ denote the \emph{Stanley--Reisner ideal} of $K(S)$ in $\bigwedge \EE^n$.
	Then $I_{\Delta_g(S)}$ is the initial ideal $I_{\Delta_g(S)} = \In(g\cdot I_S)$ with respect to the lexicographic order, where $g\cdot e_i = \sum_j g_{ij} e_j$,
	and $I_{\Delta(S)} = \Gin(I_S)$ is the generic initial ideal of $I_S$,
	given by $\Gin(I_S) = \In(g\cdot I_S)$ for any $g \in \GL(n,\EE)$ with $\FF$-algebraically independent entries.
\end{remark}
\begin{remark}[Shifting via Grassmannians]
	\label{rmk:grassmannian}
	Another perspective on algebraic shifting is as follows.
	For each $n$, $k$ and $m$, an invertible matrix $g \in \GL(n, \EE)$ determines a map $\mathit{rs}_g\colon \binom{\binom{[n]}{k}}{m} \to \Gr(m, \binom{n}{k})$
	that sends a hypergraph $S$ to the row span of $g^{\wedge S}$.
	The Plücker coordinates $p_T$ of $\Gr(m, \binom{n}{k})$ are indexed by sets $T \in \binom{\binom{[n]}{k}}{m}$.
	Algebraic shifting $\Delta_g(-)$ is the composition of $\mathit{rs}_g$ with $\Gr\bigl(m, \binom{n}{k}\bigr) \to \binom{\binom{[n]}{k}}{m}$, $V \mapsto \lex\min\Set[\big]{T \in \binom{\binom{[n]}{k}}{m}; p_T(V) \neq 0}$.
\end{remark}

\section{Bruhat decomposition}
\label{sec:bruhat}

It is clear that finitely many matrices suffice to describe all partial shifts of one $k$-uniform hypergraph on $[n]$, even if we consider algebraic shifting over an infinite field $\FF$.
This observation is immediate from the fact that there are finitely many $k$-uniform hypergraphs on the set $[n]$.
The purpose of this section is to explain how to find such matrices.
Obviously, the action of the group $\GL(n,\EE)$ on the exterior algebra $\bigwedge \EE^n$ must play a role; see~\cite[\S2.9]{Kalai02} and~\cite[Part III]{Fulton:Young+tableaux}.

Of key importance is the subgroup $B \subseteq \GL(n, \EE)$ of upper triangular matrices, which is also called the \emph{Borel subgroup} of $\GL(n, \EE)$.
It contains the \emph{unipotent subgroup} $U$ that comprises those upper triangular matrices which have ones on the diagonal.
We let the symmetric group $\SymmetricGroup{n}$ on $n$ elements act on $[n]$ from the right and write $i \cdot w$ for $i \in [n]$ and $w \in \SymmetricGroup{n}$.
This convention aligns with the notation used by \OSCAR~\cite{OSCAR-book},
but differs from writing permutations acting like functions (i.e., on the left) as, e.g., in~\cite{Fulton:Young+tableaux}.
Following the convention in~\cites[Figure~2.5]{BjornerBrenti:2005}[§1.5]{Stanley:EC1},
we regard $\SymmetricGroup{n}$ as a subgroup of $\GL(n, \EE)$ by assigning to $w \in \SymmetricGroup{n}$ the permutation matrix $(\delta_{i \cdot w,j})_{ij} \in \GL(n, \EE)$.
With matrices being multiplied to row vectors on the right, this ensures $e_i w = e_{i\cdot w}$, where $e_i$ denotes the $i$th standard basis vector of $\EE^n$.
In particular, if $g=(g_{ij})$ is an $n{\times}n$-matrix and $v, w \in \SymmetricGroup{n}$, then
\begin{equation}\label{eq:right action Sn}
	(vgw^{-1})_{ij} \ =  \ g_{i \cdot v, j \cdot w} \enspace.
\end{equation}
With this notation, the \emph{Bruhat decomposition} of $\GL(n, \EE)$ is the partition
\begin{equation}\label{eq:bruhat}
	\GL(n, \EE) \ = \ \smashoperator{\coprod_{w \in \SymmetricGroup{n}}} BwB \enspace ,
\end{equation}
into pairwise disjoint double cosets with respect to $B$; see~\cites[\S6.1.4]{Abramenko+Brown:2008}[\S30]{Bump:2013}[Example 1.2.11]{BjornerBrenti:2005}.
Note that the (left) cosets of $B$ in $\GL(n,\EE)$ bijectively correspond to the maximal flags of subspaces in the projective space $\PG(n-1,\EE)$; this is because $B$ is the stabilizer of the maximal flag formed from the subspaces $\langle e_1,e_2,\dots,e_n\rangle \supseteq \langle e_2,e_3,\dots,e_n\rangle \supseteq \ldots \supseteq\langle e_n\rangle \supseteq 0$.
We obtain an induced decomposition of the flag space
\[
\GL(n,\EE)/B \ = \ \smashoperator{\coprod_{w \in \SymmetricGroup{n}}} BwB/B \enspace,
\]
where the parts $BwB/B$ are known as the \emph{Schubert cells}.

\subsection*{Triangular and diagonal matrices}
The relevance of the Bruhat decomposition for partial shifting will become apparent after establishing a few more facts.

\begin{lemma}[{\cite[Theorems~6.12, 6.13, 6.19]{Fiedler:2013}}]
	\label{thm:compound matrix}
	If $g$ is invertible, diagonal, upper triangular, and/or lower triangular, then so is $g^{\wedge k}$.
	Furthermore, $(gh)^{\wedge k} = g^{\wedge k} h^{\wedge k}$.
\end{lemma}

\begin{lemma} 
	\label{thm:invariance upper triangular}
	If $b \in B$ is upper triangular and $g \in \GL(n, \EE)$, then $\Delta_{gb}(S) = \Delta_g(S)$ for all $S$.
\end{lemma}
\begin{proof}
	By \cref{thm:compound matrix}, $(gb)^{\wedge S} = g^{\wedge S} b^{\wedge k}$, and $b^{\wedge k}$ is upper triangular.
	Therefore,
	\[g^{\wedge S}_{*\sigma} \in \Span_{\EE}{g^{\wedge S}_{*\tau}; \tau < \sigma}\]
	if and only if
	\[(g^{\wedge S} b^{\wedge k})_{*\sigma} \in \Span_{\EE}{(g^{\wedge S} b^{\wedge k})_{*\tau}; \tau < \sigma}.\qedhere\]
\end{proof}
In particular, if $g \in \GL(n, \EE)$ has Bruhat decomposition $g = b w b'$ for $b, b' \in B$ and $w \in \SymmetricGroup{n}$,
the lemma tells us that $\Delta_{bwb'}(-) = \Delta_{bw}(-)$.
Note that the diagonal matrices form a maximal algebraic torus $T$ in $\GL(n,\EE)$.
\begin{lemma}
	\label{thm:invariance diagonal matrices}
	If $t$ is an invertible diagonal matrix, then for all $S$ we have
	\[
	\Delta_{gt}(S) \ = \ \Delta_{tg}(S) \ = \ \Delta_{g}(S) \enspace.
	\]
\end{lemma}
\begin{proof}
	As $t$ is upper triangular we get $\Delta_{gt}(S)=\Delta_{g}(S)$ from \cref{thm:invariance upper triangular}.
	The statement $\Delta_{tg}(S) = \Delta_{g}(S)$ follows because acting with a diagonal matrix first does not change the structure of the row echelon form.
\end{proof}

We can use these statements to lower the number of indeterminates necessary to compute the algebraic shift $\Delta(S)$.
Fixing $n^2$ many $\FF$-algebraically independent elements $x_{ij} \in \EE$ for $1 \leq i , j \leq n$, we define the matrix $\frX=(x_{ij})$.
This matrix occurs as $g$ in \cref{exmp:algo}.
The multiplication $z \mapsto \frX z$ with the matrix $\frX$ on the left is a generic linear automorphism of the $\FF$-vector space $\EE^n$.
Occasionally, we write $\FF(c)$ for the smallest subfield of $\EE$ containing all coefficients of a matrix $c\in\EE^{n{\times}n}$.
Then, by~\cite[Remark 9.2]{Milne:2022}, the extension $\FF(\frX) \supseteq \FF$ is purely transcendental, i.e., the field $\FF(\frX)$ is isomorphic to a field of rational functions; see also \cref{rem:monte-carlo}.
Further, we consider the unipotent matrix
\begin{equation}\label{eq:u}
	\frU \ \coloneqq \
	\begin{psmallmatrix}
		1 & x_{12} & x_{13} & \cdots & x_{1n}    \\
		& 1      & x_{23} & \cdots & x_{2n}    \\
		&        & \ddots & \ddots & \vdots    \\
		&        &        & 1      & x_{n-1,n} \\
		&        &        &        & 1
	\end{psmallmatrix} \enspace.
\end{equation}
Note that $\frU$ only uses $\tfrac{1}{2}n(n-1)$ many among the chosen $\FF$-algebraically independent elements.
We will see in \cref{prop:w0} that $\frU$ suffices to compute $\Delta(S)$.


Throughout the following let $w_0 \in \SymmetricGroup{n}$ be the involutive permutation of $[n]$ which sends $i$ to $n-i+1$; see also \cref{rem:length} below.
The next statement says that the full shifting operation is the same as partial shifting with respect to the matrix $\frU w_0$.
The key is to see that $\frX$ lies in the left coset $\frU w_0 B$.

\pagebreak
\begin{proposition}\label{prop:w0}
	For every $S \subseteq \binom{[n]}{k}$, we have $\Delta_{\frU w_0}(S) = \Delta_{\frX}(S)$.
\end{proposition}
\begin{proof}
	First, we note that $\Delta_{\frU w_0}(S) = \Delta_{u w_0}(S)$ for \emph{any} unipotent upper triangular matrix $u$ whose strictly super-diagonal entries are algebraically independent over $\FF$.
	Now, the matrix $w_0 \frX$ has an LU-decomposition $w_0 \frX = a b$ with $a$ lower triangular unipotent and $b \in B$.
	Rewriting that equality gives $\frX = u w_0 b$, where $u \coloneqq w_0 a w_0$ is unipotent upper triangular; here we use that $w_0=w_0^{-1}$ is an involution.
	\cref{thm:invariance upper triangular} then implies that $\Delta_{\frX}(S) = \Delta_{u w_0}(S)$.

	As $u$ is unipotent upper triangular, the transcendence degree of the field extension $\FF(u) \supseteq \FF$ is at most $\tfrac{1}{2}n(n-1)$.
	Similarly, $b$ is lower triangular, and so the transcendence degree of $\FF(b) \supseteq \FF$ is at most $n^2-\tfrac{1}{2}n(n-1)$.
	Now the transcendence degree of $\FF(\frX)\supseteq\FF$ was assumed to equal $n^2$, and the matrix $w_0$ is a $0/1$-matrix.
	Consequently, the equality $\frX = u w_0 b$ yields that the transcendence degrees of the two field extensions, $\FF(u) \supseteq \FF$ and $\FF(b) \supseteq \FF$, both attain their respective maximum.
	In particular, the strictly super-diagonal entries of $u$ are algebraically independent over $\FF$.
	So by the initial reasoning, $\Delta_{\frU w_0}(S) = \Delta_{u w_0}(S) = \Delta_\frX(S)$.
\end{proof}

\begin{example}\label{exmp:w0}
	We continue Example \ref{exmp:algo}, where $\FF=\GF{2}$ and $\EE = \FF(\frX)$.
	The matrix representation of the longest element $w_0$ is the matrix with ones along the anti-diagonal.
	Multiplying the unipotent matrix $\frU$ of size $4{\times}4$ by $w_0$ gives us
	\[
	\fboxsep=0pt
	\frU w_0 \ = \
	\begin{psmallmatrix}
		x_{14} & x_{13} & x_{12} & 1 \\
		x_{24} & x_{23} & 1 & 0 \\
		x_{34} & 1 & 0 & 0 \\
		1 & 0 & 0 & 0
	\end{psmallmatrix} \enspace.
	\]
	As in Example \ref{exmp:algo}, we let $S=\{\sigma,\tau\}=\{12,23\}$ and obtain
	\[
	\fboxsep=0pt
	(\frU w_0)^{\wedge S} \ = \
	\begin{psmallmatrix}
		x_{13} x_{24} + x_{14} x_{23}  & x_{12} x_{24} + x_{14} & x_{24} & x_{12} x_{23} + x_{13} & x_{23} & 1 \\
		x_{23} x_{34} + x_{24} & x_{34} & 0 & 1 & 0 & 0
	\end{psmallmatrix}\enspace.
	\]
	This $2{\times}6$-matrix has the row echelon form
	\begin{equation}\label{eq:w0:row-echelon-form}
		\renewcommand\strut{\vphantom{x^1_1}}
		\begin{psmallmatrix}
			x_{13} x_{24} + x_{14} x_{23} & x_{12} x_{24} + x_{14}                                                                                             & x_{24}                                                        & x_{12} x_{23} + x_{13}                                                                                            & x_{23}                                                        & 1                                                  \\[3.5pt]
			0                             & \substack{x_{12} x_{23} x_{24} x_{34} + x_{12} x_{24}^{2} \strut \\ + x_{13} x_{24} x_{34} + x_{14} x_{24} \strut} & \substack{x_{23} x_{24} x_{34} \strut \\ + x_{24}^{2} \strut} & \substack{x_{12} x_{23}^{2} x_{34} + x_{12} x_{23} x_{24} \strut \\ + x_{13} x_{23} x_{34} + x_{14}x_{23} \strut} & \substack{x_{23}^{2} x_{34} \strut \\ + x_{23} x_{24} \strut} & \substack{x_{23} x_{34} \strut \\ + x_{24} \strut}
		\end{psmallmatrix},
	\end{equation}
	again, the leading independent columns are indexed by $12$ and $13$.
	Hence, we obtain the same shifted hypergraph from Example \ref{exmp:algo}, that is $\Delta(S) =\Delta_{\frU w_0}(S) = \{12,13\}$.
	Comparing the matrices \eqref{eq:algo:row-echelon-form} and \eqref{eq:w0:row-echelon-form} demonstrates that \cref{prop:w0} yields a considerable computational advantage over the naive method of computing full shifts.
\end{example}

Recall the Bruhat decomposition from \eqref{eq:bruhat}.
The algebraic shift $\Delta(S) = \Delta_{\frX}(S)$ of $S \in \binom{[n]}{k}$
is the shift by a generic matrix $\frX$ in the open Bruhat cell $B w_0 B$.

\begin{definition}
	\label{def:partial-w-shift}
	For $w \in \SymmetricGroup{n}$ and $S \in \binom{[n]}{k}$, the \emph{partial shift} of $S$ with respect to the permutation $w$
	is the shift $\Delta_g(S)$ for a generic element $g \in BwB$.
\end{definition}

We stress that the partial shift of $S$ with respect to $w \in \SymmetricGroup{n}$
is not the partial shift $\Delta_w(S)$ of $S$ with respect to the permutation matrix of $w$.

We will show that this definition is indeed independent of the chosen generic element $g$; see \cref{thm:partial shift independent of rep}.
The Bruhat cell of a matrix $g \in \GL(n, \EE)$ is determined by the vanishing of certain minors of $g$.
For example, $g$ lies in the cell $B w_0 B$ if all leading principal minors of $w_0 g$ are nonzero,
$g$ lies in $B w_0 s_i B$ if precisely the leading principal $i$-minor vanishes etc.; see \cref{tab:bruhat minors}.
The correspondence between the permutation $w$ such that $g \in BwB$ and the vanishing condition on the minors of $g$ is not obvious,
and a simple description of a generic element of $B w B$ does not arise from this immediately.
Developing the latter is the aim of the following subsection.
\begin{table}
	\centering
	\(
	\begin{array}{c|cccc}
		\toprule
		w       & g_{1;1} & g_{12;12} & g_{13;12} & g_{12;13} \\ \midrule
		e       & \neq 0  & \neq 0    & *         & *         \\
		s_1     & = 0     & \neq 0    & *         & *         \\
		s_2     & \neq 0  & = 0       & \neq 0    & \neq 0    \\
		s_1 s_2 & = 0     & = 0       & \neq 0    & = 0       \\
		s_2 s_1 & = 0     & = 0       & = 0       & \neq 0    \\
		w_0     & = 0     & = 0       & = 0       & = 0       \\ \bottomrule
	\end{array}
	\)
	\caption{For a matrix $g \in \GL(n, 3)$, we have $g \in B w_0 w B$ for $w$ depending on the (non-)vanishing of the minors $g_{I;J}$ of $g$.
		Here $w_0=s_1s_2s_1=s_2s_1s_2$.
	}
	\label{tab:bruhat minors}
\end{table}

\begin{remark}
	\label{rmk:partial-shift-gin}
	In the light of \cref{rmk:gin}, we can rephrase \cref{def:partial-w-shift} as follows.
	The partial shift of $S$ with respect to $w \in \SymmetricGroup{n}$ is the complex $L$ with $I_{L} = \Gin^w(I_S)$,
	where the \emph{$w$-generic initial ideal} of $I_S$ is given by $\Gin^w(I_S) = \In(g\cdot I_S)$
	for a generic element $g \in BwB$.
\end{remark}

\subsection*{Minimal representatives}
Let $g \in \GL(n, \EE)$ be a matrix with Bruhat decomposition $g = b w b'$,
where $b, b' \in B$, and $w \in \SymmetricGroup{n}$.
The matrices $b$ and $b'$ are not uniquely determined by $g$ in general; for example,
\[
\begin{psmallmatrix}
	1 & x & y \\
	& 1 & z \\
	&   & 1
\end{psmallmatrix}
\begin{psmallmatrix}
	1 &   &   \\
	&   & 1 \\
	& 1 &
\end{psmallmatrix}
\ = \
\begin{psmallmatrix}
	1 & y & x \\
	0  & z & 1 \\
	0  & 1 & 0
\end{psmallmatrix}
\ = \
\begin{psmallmatrix}
	1 & 0 & 0 \\
	& 1 & z \\
	&   & 1
\end{psmallmatrix}
\begin{psmallmatrix}
	1 &   &   \\
	&   & 1 \\
	& 1 &
\end{psmallmatrix}
\begin{psmallmatrix}
	1 & y & x \\
	& 1 & 0 \\
	&   & 1
\end{psmallmatrix} \enspace.
\]
In the light of \cref{thm:invariance upper triangular}, we want to move as \enquote{much} of $b$ as possible to $b'$.
More precisely, our goal for this section is to find, for given $w\in\SymmetricGroup{n}$, a smallest possible set $U(w) \subseteq U$ such that $B w B = U(w) w B$.
It turns out that the entries of $w$ that can be moved to the right are in correspondence with the inversions of $w$.
For $w \in \SymmetricGroup{n}$, let
\[
\inv w \ \coloneqq \ \Set{(i,j);t\text{$i < j$ and $i \cdot w > j \cdot w$}} \enspace.
\]
be the set of \emph{inversions} of $w$.
Note that in~\cite{BjornerBrenti:2005} the notation ``$\operatorname{inv} w$'' is used for the cardinality of the set $\inv w$.
The set of inversions $\inv w$ is equivalent to what is called an ``inversion table'' in~\cite{Stanley:EC1}, and so $\inv(w)$ determines the permutation $w$~\cite[Proposition 1.3.9]{Stanley:EC1}.
For $w \in \SymmetricGroup{n}$, let
\[
U(w) \ \coloneqq \ \Set{u \in U; u_{ij} = 0 \ \text{for all}\ (i, j) \notin \inv w } \enspace.
\]

\begin{remark}\label{rem:length}
	The \emph{length} $\ell(w)$ of a $w \in \SymmetricGroup{n}$ is the minimal length of a word expressing $w$ in terms of the simple transpositions $s_1=(1 \ 2),\dotsc,s_{n-1}=(n{-}1\ n) \in \SymmetricGroup{n}$.
	With this notation we have $\ell(w) = \abs{\inv w} = \abs{U(w)}$ for all $w \in \SymmetricGroup{n}$.

	The length function gives rise to the graded order $\abs{\SymmetricGroup{n}}_q \coloneqq \sum_{w \in \SymmetricGroup{n}} q^{\ell(w)}$.
	For $n \in \NN$, let $[n]_q \coloneqq \frac{1-q^n}{1-q} = q^0 + \dotsb + q^{n-1}$ and $[n]_q! \coloneqq [n]_q \dotsm [1]_q$.
	Then $\abs{\SymmetricGroup{n}}_q = [n]_q!$.

	Because $\SymmetricGroup{n}$ is a finite Coxeter group, the symmetric group has a unique element of maximal length,
	which is precisely the order-reversing permutation $w_0$; see~\cite[\S2.3]{BjornerBrenti:2005}.
	The Bruhat cell $B w_0 B$ in the Bruhat decomposition of $\GL(n,\EE)$ is an open and dense subset in the Zariski topology.
\end{remark}

In \cref{prop:w0} we already saw that the partial shift by $\frU w_0$ agrees with the full shift.
More generally, we will see that the length of a permutation $w$ indicates how close to the full shift the partial shift with respect to $w$ is; e.g., see \cref{thm:n-cycle} below.

\begin{lemma}
	\label{thm:double coset representatives}
	For all $w \in \SymmetricGroup{n}$, we have
	\begin{enumerate}
		\item $U w B = U(w) w B$, and \label{thm:double coset representatives:1}
		\item If $\EE$ is finite, then $U(w)$ has minimal cardinality with that property. \label{thm:double coset representatives:2}
	\end{enumerate}
\end{lemma}
\begin{proof}
	To prove \proofref{thm:double coset representatives:1}, let $u \in U$.
	We will construct matrices $u' \in U(w)$ and $u'' \in U(w^{-1})$ such that $u w = u' w u''$.
	For $\gamma \in \EE$ and $k \neq l$, let $e_{kl}(\gamma)$ be the elementary matrix
	with entries $e_{kl}(\gamma) = \delta_{ij} + \gamma\delta_{ik}\delta_{lj}$.
	Clearly, $e_{kl}(\gamma)^{-1} = e_{kl}(-\gamma)$, and by \eqref{eq:right action Sn}, we have $w^{-1} e_{kl}(\gamma) w = e_{k \cdot w, l \cdot w}(\gamma)$.
	For $k < l$, define the matrices $u' \coloneqq u \: e_{kl\mathstrut}(-u_{kl})$ and $u'' \coloneqq e_{k \cdot w, l \cdot w}(u_{kl})$.
	Since $u, e_{kl}(-) \in U$, we also have $u' \in U$.
	Furthermore, we have $u'_{kl} = 0$, and if $(k,l) \notin \inv w$, then $u'' \in U$.
	We obtain
	\[
	  u' \: w \: u''
	\ =\ u \: e_{kl\mathstrut}(-u_{kl}) \: w \: e_{k \cdot w, l \cdot w}(u_{kl})
	\ =\ u \: e_{kl\mathstrut}(-u_{kl}) \: w \: w^{-1} \: e_{kl\mathstrut}(u_{kl}) \: w
	\ =\ u \: w.
	\]
	as desired.
	Repeating this for all $(k < l) \notin \inv w$ proves \proofref{thm:double coset representatives:1}.
	In particular, one gets Bruhat decomposition $\GL(m, \EE) = \coprod_{w \in \SymmetricGroup{n}} U(w) w B$,
	with $B = TU$.

	To prove \proofref{thm:double coset representatives:2}, assume that $\EE \cong \GF{q}$.
	It is immediate from the definition of $U(w)$ that $\abs{U(w)} = q^{\ell(w)}$.
	Furthermore, it is standard that
	\begin{equation*}
		\abs{U} = q^{\frac{1}{2}n(n-1)},\qquad
		\abs{T} = (q-1)^n,\qquad
		\abs{\GL(n, \GF{q})} = \prod_{k=0}^{n-1}(q^n-q^k).
	\end{equation*}
	Then
	\begin{multline*}
			\smashoperator{\sum_{w \in \SymmetricGroup{n}}} \abs{U(w)} \abs{B}
			\ = \ \smashoperator{\sum_{w \in \SymmetricGroup{n}}} q^{\ell(w)} \abs{T} \abs{U}
			\ = \ \biggl(\prod_{k=1}^{n} \frac{q^k-1}{q-1}\biggr) (q-1)^n q^{\frac{1}{2}n(n-1)} \\
			\ = \ \biggl(\prod_{k=1}^{n} (q^k-1)\biggr) q^{\frac{1}{2}n(n-1)}
			\ = \ \prod_{k=1}^{n} (q^k-1)q^{n-k-1}
			\ = \ \prod_{k=0}^{n-1}(q^n-q^k)\\
			\ = \ \abs{\GL(n, \GF{q})} \enspace.
	\end{multline*}
	Together with $\GL(m, \EE) = \coprod_{w \in \SymmetricGroup{n}} U(w) w B$, this proves \proofref{thm:double coset representatives:2}.
\end{proof}

\begin{corollary}
	\label{thm:partial shift independent of rep}
	If $g, g'$ are generic elements of the Bruhat cell $BwB$, then $\Delta_g(S) = \Delta_{g'}(S)$ for all $S$.
\end{corollary}
\begin{proof}
	\Cref{thm:double coset representatives:1} shows that $g = uwb$ and $g' = u'wb'$
	for $u, u' \in U(w)$ and $b, b' \in B$.
	By \cref{thm:invariance upper triangular}, $\Delta_g(S) = \Delta_{uw}(S)$, and analogously for $g'$.
	The entries $\Set{u_{ij}; (i < j) \in \inv w}$ are algebraically independent over $\FF$.
	Replacing these by the algebraically independent entries $\Set{u'_{ij}; (i < j) \in \inv w}$ does not alter the partial shift.
\end{proof}

Our further analysis will make use of special unipotent matrices.
The idea of the following definition is that $\frU(w) \in U(w)$ is for $w \in \SymmetricGroup{n}$
what $\frU \in U$ is for $w_0 \in \SymmetricGroup{n}$.
Recall that the coefficients of the matrix $\frX=(x_{ij})\in\EE^{n\times n}$ are chosen to be $\FF$-algebraic independent.
The following generalizes the definition of $\frU$ in \eqref{eq:u}:

\begin{definition}
	\label{def:u(w)}
	For $w \in \SymmetricGroup{n}$, define the $n \times n$-matrices $\frU(w) \in U(w)$ by
	\begin{equation*}
		\frU(w) \coloneqq
		\begin{psmallmatrix}
			1 & u_{12} & \cdots & u_{1n}    \\
			& \ddots & \ddots & \vdots    \\
			&        & 1      & u_{n-1,n} \\
			&        &        & 1
		\end{psmallmatrix}
		\quad \text{with} \quad
		u_{ij} = \begin{cases*}
			x_{ij} & if $(i,j) \in \inv w$,\\
			0 & otherwise,
		\end{cases*}
	\end{equation*}
	and let $\frR(w) \coloneqq \frU(w)w$.
\end{definition}
The matrix $\frR(w)$ can be viewed as a \emph{Rothe diagram} of $w$,
with a $1$ in position $(i, i\cdot w)$ for every $i$, and an $x_{ij}$ in position $(i, j\cdot w)$ for every $(i,j) \in \inv w$;
cf.~\cite[14]{Knuth:1998}.
Note that the matrix $\frU$ from \eqref{eq:u} is just $\frU = \frU(w_0)$.
If $w=s_i$ is a simple transposition, we have $\inv s_i=\{(i,i+1)\}$, and $\frU(w)$ is a transvection.
These transvections correspond to elementary row operations and generate the entire unipotent subgroup $U$.

\begin{remark}
	Recall from \cref{prop:w0} that $\Delta_{\frR(w_0)}(S) = \Delta_{\frX}(S)$.
	In this way, \cref{def:u(w)} generalizes $\frU$ to \enquote{less generic} cases.
	More precisely, the matrix $\frR(w)$ is generic within the Bruhat cell $U(w) w B = BwB$.
\end{remark}

As it turns out, the fixed matrix $\frU$, together with all the permutations, suffices to compute all partial shifts.
The following lemma is instrumental by providing a criterion for a map to preserve algebraic independence.

\begin{lemma}[{\cite[Theorem~2.2]{EhrenborgRota:1993}}]
	\label{thm:algebraic independence jacobian}
	If $b_1,\dotsc,b_n$ is a transcendendence basis of $\EE \supseteq \FF$,
	then the elements $p_1,\dotsc,p_n \in \EE$ are $\FF$-algebraically independent
	if and only if not all maximal minors of the Jacobian $(\partial p_j/\partial b_i)_{ij}$ are identically zero.
\end{lemma}

\begin{proposition}
	\label{thm:partial shift by u(w)}
	The partial shift of $S$ by $w$ is given by $\Delta_{\frR(w)}(S) = \Delta_{\frU w}(S)$.
\end{proposition}

\begin{proof}
	By \cref{thm:double coset representatives:1}, we have $\frU w \in \frU' w B$ for a matrix $\frU' \in U(w)$;
	in particular, $\Delta_{\frU w(S)} = \Delta_{\frU' w}(S)$.
	To make the matrix $\frU'$ explicit,
	enumerate the set $\Set{(k < l) \notin \inv w; k, l \in [n]} \eqqcolon \{(k_1, l_1),\dotsc,(k_r, l_r)\}$.
	According to the proof of \cref{thm:double coset representatives:1},
	we have
	\[
	\frU \: w
	\ =\ \underbrace{\frU \prod_{s = 1}^r e_{k_s l_s}(-\gamma_s)}_{\eqqcolon \frU'}
	\: w \:
	\underbrace{\prod_{s = r}^1 e_{k_s \cdot w, l_s \cdot w}(\gamma_s)}_{\in U} \enspace,
	\]
	where $\gamma_1 \coloneqq x_{k_1 l_1}$ is the $(k_1, l_1)$-entry of $\frU$,
	and $\gamma_s$ is the $(k_s, l_s)$-entry of the matrix product $\frU \prod_{t = 1}^{s-1} e_{k_t l_t}(-\gamma_t)$ for $s > 1$.
	Let $p_{ij} \in \EE$ be the $(i,j)$-entry of $\frU'$.

	\begin{claim*}
		The non-zero super-diagonal entries $\Set{p_{kl}; (k < l) \in \inv w}$ of $\frU'$ are algebraically independent over $\FF$.
	\end{claim*}
	\begin{claimproof}
		Without loss of generality, let $r = 1$ and $(k, l) \coloneqq (k_1, l_1)$.
		In this case,
		\[
		\frU' = \begin{pNiceArray}[small]{*{11}{>{\color{gray}}c}}
			1 & x_{12} & \cdots             & x_{1,k-1} & x_{1,k}   & x_{1,k+1}   & \cdots             & x_{1,l-1}   & \color{black} x_{1,l} - x_{1,k}x_{k,l}     & x_{1,l+1}    & \cdots             \\
			& 1      & \cdots             & x_{2,k-1} & x_{2,k}   & x_{2,k+1}   & \cdots             & x_{2,l-1}   & \color{black} x_{2,l} - x_{2,k}x_{k,l}     & x_{2,l+1}    & \cdots             \\
			&        & \Ddots[color=gray] & \vdots    & \vdots    & \vdots      &                    & \vdots      & \vdots                                     & \vdots       &                    \\
			&        &                    & 1         & x_{k-1,k} & x_{k-1,k+1} & \cdots             & x_{k-1,l-1} & \color{black} x_{k-1,l} - x_{k-1,k}x_{k,l} & x_{ k-1,l+1} & \cdots             \\
			&        &                    &           & 1         & x_{k,k+1}   & \cdots             & x_{k,l-1}   & \color{black} 0                            & x_{ k,l+1}   & \cdots             \\
			&        &                    &           &           & 1           & \cdots             & x_{k+1,l-1} & x_{k+1,l}                                  & x_{ k+1,l+1} & \cdots             \\
			&        &                    &           &           &             & \Ddots[color=gray] & \vdots      & \vdots                                     & \vdots       & \Ddots[color=gray]
		\end{pNiceArray},
		\]
		where the gray entries are the same as in $\frU$.
		Ordering the expressions $p_{ij}$ by the indices $(i,j)$ such that $(i,j) \preceq (i',j')$ if $j < j'$, or $j = j'$ and $i \geq i'$,
		we obtain that the Jacobian of the entries $p_{ij}$ (except for $p_{kl} = 0$) with respect to the $x_{ij}$ is
		\[
		\left(\frac{\partial p_{jj'}}{\partial x_{ii'}}\right)_{ii';jj'} =
		\begin{NiceArray}[small,nullify-dots,margin]{c @{\hspace{.75em}} ccc wc{7em} *{10}{c} @{\hspace{.75em}} c}
			& p_{12} & p_{23} & p_{13} & \Cdots & & & & & p_{k+1,l} & p_{k-1,l}  & \Cdots & p_{1,l}  & p_{l,l+1} & \cdots & \\
			x_{12}    & 1      &        &        &        & & & & &           &            &        &          &           &        & \\
			x_{23}    &        & 1      &        &        & & & & &           &            &        &          &           &        & \\
			x_{13}    &        &        & 1      &        & & & & &           &            &        &          &           &        & \\
			\Vdots    &        &        &        & \Ddots & & & & &           &            &        &          &           &        & \\
			&        &        &        &        & & & & &           & 1          &        &          &           &        & x_{k-1,k} \\
			&        &        &        &        & & & & &           &            & \Ddots &          &           &        & \Vdots \\
			&        &        &        &        & & & & &           &            &        & 1        &           &        & x_{1,k} \\
			&        &        &        &        & & & & &           &            &        &          &           &        & \\
			x_{k+1,l} &        &        &        &        & & & & &  1        &            &        &          &           &        & \\
			\Vdots    &        &        &        &        & & & & &           & -p_{k-1,k} & \cdots & -p_{1,k} &           &        & x_{k,l}\\
			&        &        &        &        & & & & &           & 1          &        &          &           &        & x_{k-1,l} \\
			&        &        &        &        & & & & &           &            & \Ddots &          &           &        & \Vdots   \\
			&        &        &        &        & & & & &           &            &        & 1        &           &        & x_{1,l}  \\
			&        &        &        &        & & & & &           &            &        &          & 1         & {}     & x_{l,l+1}\\
			&        &        &        &        & & & & &           &            &        &          &           & \Ddots & \Vdots
			\CodeAfter
			\SubMatrix{()}{2-2}{16-15}{)}
		\end{NiceArray}.
		\]
		The above matrix has linearly independent columns.
		Since $x_{12},\dotsc,x_{n-1,n}$ is a transcendence basis of $\EE \supseteq \FF$,
		\cref{thm:algebraic independence jacobian} implies that $p_{12}, \dotsc, \hat{p}_{kl},\dotsc,p_{n-1,n}$ are algebraically independent,
		where $\hat{p}_{kl}$ is omitted.
		The claim (for general $r$) now follows inductively.
	\end{claimproof}

	By the same argument as in the proof of \cref{prop:w0},
	$\Delta_{\frR(w)}(S) = \Delta_{u w}(S)$ for \emph{any} matrix $u \in U(w)$
	whose entries $\Set{u_{ij}; (i < j) \in \inv w}$ are algebraically independent.
	In particular, $\Delta_{\frR(w)}(S) = \Delta_{\frU' w}(S) = \Delta_{\frU w}(S)$.
\end{proof}

The following example shows that the genericity assumption in \cref{thm:partial shift independent of rep} is crucial.
\begin{example}\label{exmp:vandermonde}
  In personal communication, Gil Kalai raised the question about the behavior of partial shifting $\Delta_\frM(-)$ with respect to the Vandermonde matrix
  \begin{equation}\label{eq:vandermonde}
    \frM \ = \ \begin{pmatrix}
      x_1    & \cdots & x_n    \\
      \vdots &        & \vdots \\
      x_1^n  & \cdots & x_n^n
    \end{pmatrix}
  \end{equation}
  for $\FF$-algebraically independent $x_1,\dotsc,x_n \in \EE$.
  It turns out that $\frM$ lies in the $Bw_0B$-cell of $\GL(n, \EE)$, but the transcendence degree of $\FF(\frM)$ is only $n$.
  So $\frM$ does not satisfy the assumptions of \cref{thm:partial shift independent of rep}.
  In fact, the $3$-uniform hypergraph $S=\{123, 145, 246, 356\}$ shows that for $\FF = \QQ$,
  we have $\Delta(S) = \{123, 124, 125, 126\}$, while $\Delta_\frM(S) = \{123, 124, 125, 134\}$.
  Notably, $\Delta_\frM(S) \neq \Delta_{\frR(w)}(S)$ for \emph{all} $w \in \SymmetricGroup{6}$.
\end{example}

The columns of the Vandermonde matrix $\frM$ in \eqref{eq:vandermonde} form points on the moment curve in $\FF(\frM)^n$.
If $\FF$ is an ordered field such as, e.g., $\FF=\QQ$, then $\FF(\frM)$ can be equipped with an ordering, too.
In that case, the convex hull of finitely many points on the moment curve is known as a cyclic polytope.
Murai \cite{Murai07} investigated the full (exterior and symmetric) shifts of the (simplicial) boundary complexes of those polytopes.

\subsection*{The unipotent matrices \texorpdfstring{\boldmath$\frR(w)$}{𝔯(w)}}
In this section, we collect properties of the matrices $\frR(w)$.
Specifically, we want to use this information to understand how $\frR(v)$ and $\frR(w)$ are related to $\frR(vw)$.
The central result is \cref{thm:u(ws) and u(w)u(s)},
which will be used in \cref{sec:partial-shift-graph} to define the \emph{partial shift graph} and prove its acyclicity.
For the following lemma, recall that we let $\SymmetricGroup{n}$ act on $[n]$ from the right. The symbol $\symdiff$ denotes the symmetric difference.

\begin{lemma}
	\label{thm:inv(vw)}
	We have $\inv vw = \inv v \symdiff (\inv w) \cdot v^{-1}$ for every $v$, $w \in \SymmetricGroup{n}$.
	In particular, $\ell(vw) = \ell(v) + \ell(w)$ if and only if $\inv v \cap (\inv w) \cdot v^{-1} = \emptyset$,
	and in this case, $\inv vw = \inv v \cup (\inv w) \cdot v^{-1}$.
\end{lemma}
The condition $\ell(vw) = \ell(v) + \ell(w)$ is equivalent to $vw$ being larger or equal to $v$ in the right weak order on $\SymmetricGroup{n}$; see~\cite[Prop.\ 3.1.2(ii)]{BjornerBrenti:2005}.
	For an $n \times n$-matrix $m$ and $J \subseteq [n] \times [n]$,
	let $m|_J$ be obtained from $m$ by setting the entries $m_{ij}$ with $(i,j) \notin J$ to zero.
	Note that
	\begin{equation}\label{eq:u(w)-res}
		\frU(w) = e + \frX|_{\inv w}.
	\end{equation}
	For a permutation $w \in \SymmetricGroup{n}$ and a matrix $m$ with entries in $\FF[x_{ij} \mid 1 \leq i,j \leq n]$,
	let $m^w$ be the matrix defined entry-wise by $(x_{ij})^w \coloneqq x_{i \cdot w^{-1}, j \cdot w^{-1}}$.
	Note that $\frX^w = w^{-1} \frX w$, 
	and thus by \eqref{eq:u(w)-res}, we get that
	\begin{equation}\label{eq:u(w)^v}
		\frR(w)^v = (e + (v^{-1} \frX v)|_{\inv w})w.
	\end{equation}

\begin{lemma}
	\label{thm:x(v)x(w)}
	If $v$, $w \in \SymmetricGroup{n}$ with $\ell(vw) = \ell(v)+\ell(w)$,
	then
	\begin{equation}
		\label{eq:x(v)x(w)}
		\frR(v) \: \frR(w)^v \ =\ \frR(vw) \; + \; \Bigl(\smashoperator{\sum_{\subalign{j\colon (i, j \cdot v^{-1}) &\in \inv v,\\ (j, k \cdot w^{-1}) &\in \inv w}}} x_{i, j \cdot v^{-1}} x_{j \cdot v^{-1}, k \cdot (vw)^{-1}} \Bigr)_{ik} \enspace .
	\end{equation}
\end{lemma}

The reader should not worry too much about the $v$ in $\frU(w)^v$;
this is just to make sure that the indices of the indeterminates $x$ in $\frU(v)$ and $\frU(w)^v$ match those
occurring in $\frU(vw)$.

\begin{proof}
	First, we note that for any matrix $m$ and $J \subseteq [n]\times [n]$, we have
	\begin{equation}
		\label{thm:x(v)x(w):claim J}
		m|_{J \cdot v^{-1}} \ = \ v \: (v^{-1} m v)|_J \: v^{-1}
	\end{equation}
	because
	\begin{multline*}
		\bigl(v \: (v^{-1} m v)|_J \: v^{-1}\bigr)_{ij}
		\stackrel{\eqref{eq:right action Sn}}{=} \bigl((v^{-1} m v)|_J\bigr)_{i \cdot v,j \cdot v}
		= \begin{smallcases}
			(v^{-1} m v)_{i \cdot v, j \cdot v} & \text{if $(i \cdot v, j \cdot v) \in J$,} \\
			0 & \text{otherwise}
		\end{smallcases}\\
		\stackrel{\eqref{eq:right action Sn}}{=} \begin{smallcases}
			m_{ij} & \text{if $(i, j) \in J \cdot v^{-1}$,} \\
			0 & \text{otherwise}
		\end{smallcases}
		= (m|_{J \cdot v^{-1}})_{ij} \enspace.
	\end{multline*}
	Now, we get that
	\begin{align}
		\frR(v) \: \frR(w)^{v}
          \ &\stackrel{\mathclap{\eqref{eq:u(w)^v}}}{=} \ (v + \frX|_{\inv v} v) \: (w + (v^{-1} \frX v)|_{\inv w}w) \notag \\
            &= \
              \underbrace{vw + \frX|_{\inv v} vw + v \: (v^{-1} \frX v)|_{\inv w} \: v^{-1} v w}_{(*)}
              + \underbrace{ \frX|_{\inv v}v \: (v^{-1} \frX v)|_{\inv w} w }_{(**)} \label{eq:proof:u(v)u(w):1} \enspace.
	\end{align}
	We analyze the two braces separately.
	For the first one, we get
	\begin{equation*}
		(*) \
		\stackrel{\eqref{thm:x(v)x(w):claim J}}{=} \ \bigl(e + \underbrace{ \frX|_{\inv v} + \frX|_{(\inv w) \cdot v^{-1}}}_{ \frX|_{\inv v \cup \inv w \cdot v^{-1}} } \bigr) vw
		\ \stackrel{\text{\cref{thm:inv(vw)}}}{=} \ (e + \frX|_{\inv vw}) vw
		\ \stackrel{\mathclap{\eqref{eq:u(w)-res}}}{=} \ \frR(vw) \enspace.
	\end{equation*}
	For $(**)$, we get that
	\def\TempA{\scriptstyle x_{j \cdot v^{-1},k \cdot (vw)^{-1}}}
	\def\TempB{\scriptsize if $(j, k \cdot w^{-1}) \in \inv w$,}
	\begin{align*}
		\frX|_{\inv v}v &= \left(
		\begin{smallcases}
			x_{i,j \cdot v^{-1}}                   & \makebox[\widthof{\TempB}][l]{\scriptsize if $(i,j \cdot v^{-1}) \in \inv v$,} \\
			\mathmakebox[\widthof{$\TempA$}][l]{0} & \text{otherwise},
		\end{smallcases}
		\right)_{ij},
		\\
		(v^{-1} \frX v)|_{\inv w}w &= \left(
		\begin{smallcases}
			x_{j \cdot v^{-1},k \cdot (vw)^{-1}} & \text{if $(j, k \cdot w^{-1}) \in \inv w$,} \\
			0                                    & \text{otherwise},
		\end{smallcases}
		\right)_{jk}
	\end{align*}
	by unraveling the definitions,
	so the matrix product $(**)$ gives the sum in \cref{eq:x(v)x(w)}.
\end{proof}

\begin{corollary}
	For all $i \neq j$, we have $\frR(s_i) \frR(s_j)^{s_i} = \frR(s_is_j)$.
\end{corollary}
\begin{proof}
	In this case, the parenthesized sum in \cref{eq:x(v)x(w)} ranges over the empty set.
\end{proof}

\begin{example}
	With $n = 4$ and $k = 2$, a case of shortest length with $\frR(v) \: \frR(w)^v \neq \frR(vw)$ is
	$v = s_1 s_2$ and $w = s_1$.
	In this case, we obtain
	\begin{align*}
		\frR(vw) &= \begin{psmallmatrix}
			x_{13} & x_{12} & 1 & 0 \\
			x_{23} & 1      & 0 & 0 \\
			1      & 0      & 0 & 0 \\
			0      & 0      & 0 & 1
		\end{psmallmatrix},&
		\frR(v) \frR(w)^v &= \begin{psmallmatrix}
			\bm{x_{12} x_{23}}+ x_{13} & x_{12} & 1 & 0 \\
			x_{23}                     & 1      & 0 & 0 \\
			1                          & 0      & 0 & 0 \\
			0                          & 0      & 0 & 1
		\end{psmallmatrix},
		\\
		\frR(vw)^{\wedge k} &= \begin{psmallmatrix}
			\substack{\bm{-x_{12} x_{23}} \\ + x_{13}} & -x_{23} & 0      & -1 & 0      & 0 \\[2pt]
			-x_{12}                                    & -1      & 0      & 0  & 0      & 0 \\
			0                                          & 0       & x_{13} & 0  & x_{12} & 1 \\
			-1                                         & 0       & 0      & 0  & 0      & 0 \\
			0                                          & 0       & x_{23} & 0  & 1      & 0 \\
			0                                          & 0       & 1      & 0  & 0      & 0
		\end{psmallmatrix}, &
		\bigl(\frR(v) \frR(w)^v\bigr)^{\wedge k} &= \begin{psmallmatrix}
			x_{13}  & -x_{23} & 0                                        & -1 & 0      & 0 \\
			-x_{12} & -1      & 0                                        & 0  & 0      & 0 \\[2pt]
			0       & 0       & \substack{\bm{x_{12} x_{23}}\\ + x_{13}} & 0  & x_{12} & 1 \\[2pt]
			-1      & 0       & 0                                        & 0  & 0      & 0 \\
			0       & 0       & x_{23}                                   & 0  & 1      & 0 \\
			0       & 0       & 1                                        & 0  & 0      & 0
		\end{psmallmatrix}
	\end{align*}
	with the offending entries in bold.
\end{example}

\newcommand{\ColMat}{\mathcal{M}}
\newcommand{\IndMat}{\mathcal{I}}
The \emph{column matroid} $\ColMat(m)$ of an $k \times n$-matrix $m$
is the matroid with ground set $[n]$, whose
independent sets are the sets $J$ of column indices of $m$ for which the columns $(m_{*j})_{j\in J}$ are linearly independent.
Note that as subsets of $[n]$, independent sets of $\ColMat(m)$ can be ordered lexicographically.
In particular, if $S \subseteq \binom{[n]}{k}$,
then by definition, $\Delta_g(S)$ is the lexicographically minimal basis of $\ColMat(g^{\wedge S})$,
where the ground set $\binom{[n]}{k}$ of $\ColMat(g^{\wedge S})$ itself is ordered lexicographically.

\begin{proposition}
	\label{thm:u(ws) and u(w)u(s)}
	Let $v$, $w \in \SymmetricGroup{n}$ with $\ell(vw) = \ell(v)+\ell(w)$.
	Then $\ColMat(\frR(v w)^{\wedge S}) = \ColMat\bigl((\frR(v) \: \frR(w)^v)^{\wedge S}\bigr)$
	for every $S \subseteq \binom{[n]}{k}$.
	In particular, $\Delta_{\frR(v w)}(S) = \Delta_{\frR(v) \: \frR(w)^v}(S)$.
\end{proposition}

\begin{proof}
	It suffices to prove the \lcnamecref{thm:u(ws) and u(w)u(s)} for $w = s_l$ for some $l$;
	the statement then follows inductively.
	Recall that $\frR(vw) = (e + \frX|_{\inv vw})vw$, where the entries $x_{ij}$ of $\frX$ are algebraically independent.
	We will show that $\frR(v) \frR(w)^v = (e + \frY|_{\inv vw})vw$
	for a different matrix $\frY$ with entries in $\EE$ algebraically independent over $\FF$.
	Of course, $\ColMat\bigl((e + \frX|_{\inv vw})vw\bigr)$ does not depend on the precise collection $\frX$
	of $\FF$-algebraically independent elements,
	and thus does not change when replacing $\frX$ by $\frY$:
	\[
	\ColMat\bigl((e + \frX|_{\inv vw})vw\bigr) = \ColMat\bigl((e + \frY|_{\inv vw})vw\bigr).
	\]
	By \cref{thm:x(v)x(w)}, the matrix $\frR(v) \: \frR(s_l)^v$ has entries
	\begin{multline}
		\label{eq:proof:matroid-u(w)u(s):1}
		\bigl( \frR(v) \: \frR(s_l)^v \bigr)_{ik} \ = \
		\delta_{i \cdot vs_l, k} +
		\begin{cases}
			x_{i, k \cdot (v s_l)^{-1}} & \text{if $(i, k \cdot (v s_l)^{-1}) \in \inv v s_l$,}\\
			0 & \text{otherwise}
		\end{cases}
		\\ +
		\begin{cases}
			x_{i, k \cdot v^{-1}} x_{k \cdot v^{-1}, k \cdot (v s_l)^{-1}} & \text{if $(i, k \cdot v^{-1}) \in \inv v$ and $k = l$,} \\
			0 & \text{otherwise,}
		\end{cases}
	\end{multline}
	because the only inversion of $s_l$ is $(l, l+1)$, which implies $j = k = l$ is the only summand in the sum in \cref{eq:x(v)x(w)}.

	\begin{claim*}
		If $\ell(v s_l) > \ell(v)$ and $(i, l \cdot v^{-1}) \in \inv v$, then $(i, l \cdot (vs_l)^{-1}) \in \inv vs_l$.
	\end{claim*}
	\begin{claimproof}
		Let $(i, l \cdot v^{-1}) \in \inv v$; that is $i < l \cdot v^{-1}$ and $i \cdot v > l$.
		Because $\ell(vs_l) > \ell(v)$, we have $(l, l+1) \notin \inv v^{-1}$; that is, $l \cdot v^{-1} < (l+1) \cdot v^{-1}$.
		Together, we get that $i < (l+1) \cdot v^{-1}$.
		In particular, $i \cdot v \neq l+1$, so with $i \cdot v > l$, we get that $i \cdot v > l+1$.
		This shows that $(i, (l+1) \cdot v^{-1}) = (i, l \cdot (vs_l)^{-1}) \in \inv v$.
		Now $\inv v \subseteq \inv vs_l$ by \cref{thm:inv(vw)},
		which finishes the claim.
	\end{claimproof}
	We may now rewrite \cref{eq:proof:matroid-u(w)u(s):1} as
	\begin{equation*}
		\bigl( \frR(v) \: \frR(s_l)^v \bigr)_{ik} =
		\begin{cases}
			1                            & \text{if $k = i \cdot vs_l$,} \\
			x_{i, k \cdot (v s_l)^{-1}}  & \text{if $(i, k \cdot (v s_l)^{-1}) \in \inv v s_l$ and $k \neq l$,}\\
			\begin{multlined}[b]
				x_{i, k \cdot (v s_l)^{-1}} \\[-1em]
				+ x_{i, k \cdot v^{-1}} x_{k \cdot v^{-1}, k \cdot (v s_l)^{-1}}
			\end{multlined}              & \text{if $(i, k \cdot v^{-1}) \in \inv v$ and $k = l$,}\\
			0 & \text{otherwise.}
		\end{cases}
	\end{equation*}
	Recall that $\frU(w) = e + \frX|_{\inv w}$ for the generic matrix $\frX$.
	Let $\frY = (y_{ij})_{ij}$ be the matrix with entries
	\[
	y_{ij} \coloneqq \begin{cases}
		x_{ij} & \text{if $j \cdot vs_l \neq l$,} \\
		x_{ij} + x_{i, j \cdot s_l} x_{j \cdot s_l, j} & \text{if $j \cdot vs_l = l$ and $(i,j \cdot s_l) \in \inv v$,}
	\end{cases}
	\]
	and let $\frV(v) \coloneqq e + \frY|_{\inv v}$ be defined analogously to $\frU(v)$ with $\frY$ instead of $\frX$.
	With $j = k \cdot (vs_l)^{-1}$, we see that $\frR(v) \: \frR(s_l)^v = (e + \frY|_{\inv vs_l}) v s_l = \frV(vs_l)vs_l$.
	We see that as polynomial in the $x_{ij}$, the Jacobian determinant $\det(\partial y_{jj'} / \partial x_{ii'})_{ii',jj'}$
	has constant term one and thus cannot be identically zero.
	\Cref{thm:algebraic independence jacobian} now implies that the variables $y_{ik}$ are algebraically independent over $\FF$.
	Therefore, $\frU(vs_l)^{\wedge S}$ and $\frV(vs_l)^{\wedge S} = (\frR(v) \frR(s_l)^v)^{\wedge S}$ have the same column matroid for any $S$,
	which implies $\Delta_{\frR(vs_l) }(S) = \Delta_{\frV(vs_l) vs_l}(S) = \Delta_{\frR(v) \: \frR(s_l)}(S)$.
\end{proof}

\begin{remark}
	\label{rmk:shift does not commute with product}
	It is not necessarily true that $\Delta_{\frR(v) \frR(w)}(S) = \Delta_{\frR(v)}(\Delta_{\frR(w)}(S))$.
	In fact,~\cite[\S 6.2]{Kalai02} contains an example that shows that $\Delta(S) = \Delta_{\frR(w_0)}(S)$
	need not equal $\Delta_{\frR(v)}(\Delta_{\frR(w)}(S))$ for \emph{any} $v, w \in \SymmetricGroup{n}$ with $vw = w_0$ and $v, w \neq e$.
\end{remark}

\section{Combinatorial shifting}
\label{sec:combinatorial}
We now look into combinatorial shifting as defined by Erd\H{o}s--Ko--Rado~\cite{ErdosKoEtAl:1961}.
The reader is referred to the surveys by Frankl~\cite{Frankl:1987} and Kalai~\cite[\S 6.2]{Kalai02} for the connection with extremal combinatorics.
Knutson~\cite{Knutson:1408.1261} gives an algebraic-geometric interpretation of combinatorial shifting in the setting of Schubert calculus.
In particular, he relates combinatorial shifting to the equivariant cohomology of $\Gr(k,\CC^n)$; see also \cref{rmk:grassmannian}.
Woodroofe~\cite{Woodroofe:2022} connects Knutson's results back to extremal combinatorics.
In this way he obtains an exterior-algebra analogue of the Erd\H{o}s--Ko--Rado theorem.
Bulavka--Gandini--Woodroofe~\cite{BulavkaGandiniWoodroofe:2406.17857} combine combinatorial and algebraic shifting to prove further analogs to the Erd\H{o}s--Ko--Rado theorem
for intersecting families in simplicial complexes.
Hibi--Murai~\cite{Murai+Hibi:2009} carried out an extensive study of the possible Betti tables of the Stanley--Reisner ideals $I_{K'}$ for complexes $K'$ obtained from a fixed $K$ by combinatorial shifting.

\begin{definition}[{\cite[\S 2]{Frankl:1987}}]
	For $w \in \SymmetricGroup{n}$ a transposition, the \emph{combinatorial shift} $\CShift_w(S)$
	of a $k$-uniform hypergraph $S \subseteq \binom{[n]}{k}$
	is the $k$-uniform hypergraph $\CShift_w(S) \coloneqq \Set{\CShift_w(\sigma, S); \sigma \in S}$, where
        \[
          \CShift_w(\sigma, S) \ \coloneqq \ \begin{cases}
                                               \sigma \cdot w & \text{if $\sigma > \sigma \cdot w \notin S$} \\
                                               \sigma         & \text{otherwise.}
                                             \end{cases}
                                           \]
\end{definition}
By construction we have $\abs{\CShift_w(S)} = \abs{S}$.
Now we can realize any combinatorial shift as a partial shift.
\begin{proposition}
	\label{thm:partial shift of simple transpositions}
	Let $S \subseteq \binom{[n]}{k}$ and $w = (i\ j) \in \SymmetricGroup{n}$ be a transposition with $i < j$.
	\begin{enumerate}
		\item\label{thm:partial shift of simple transpositions:1}
		We have $\Delta_{\gamma(w)}(S) = \CShift_w(S)$ for the $n\times n$-matrix
		\begin{equation}
			\label{eq:cshift-matrix}
			\gamma(w) \;=\;
			\begin{pNiceMatrix}[first-col,first-row,small]
				  &   &        & i      &   &        &   & j &        &   \\
				  & 1 &        &        &   &        &   &   &        &   \\
				  &   & \ddots &        &   &        &   &   &        &   \\
				i &   &        & x_{ij} &   &        &   & 1 &        &   \\
				  &   &        &        & 1 &        &   &   &        &   \\
				  &   &        &        &   & \ddots &   &   &        &   \\
				  &   &        &        &   &        & 1 &   &        &   \\
				j &   &        & 1      &   &        &   & 0 &        &   \\
				  &   &        &        &   &        &   &   & \ddots &   \\
				  &   &        &        &   &        &   &   &        & 1
			\end{pNiceMatrix}\;.
		\end{equation}
		\item\label{thm:partial shift of simple transpositions:2}
			If $w$ is a simple transposition, then $\Delta_{\frR(w)}(S) = \CShift_w(S)$.
			In particular, $\Delta_{\frR(w)}(S) \lex\leq S$.
	\end{enumerate}
\end{proposition}
\begin{proof}
	For \proofref{thm:partial shift of simple transpositions:1}, recall that
	\[
		\Delta_{\gamma(w)}(S) \ = \ \Set[\big]{\sigma \in \tbinom{[n]}{k}; \gamma(w)^{\wedge S}_{*\sigma} \notin \Span_{\FF}[\big]{\gamma(w)^{\wedge S}_{*\tau}; \tau < \sigma}} \enspace,
	\]
	Note that $\sigma \cdot w > \sigma$ if $i \in \sigma \not\ni j$, and $\sigma \cdot w < \sigma$ if $i \notin \sigma \ni j$.
	The rows of $\gamma(w)^{\wedge k}$ are of the form
	\begin{equation}
		\label{eq:u(s_i) compound mat}
		\gamma(w)^{\wedge k}_{\sigma*} \ = \
		\def\LeaderA{\mathmakebox[\widthof{$\dotsc, x_{i,j},$}]{\dotfill}}
		\def\LeaderB{\mathmakebox[\widthof{$0,\dotsc$}]{\dotfill}}
		\left\{
		\begin{array}{r<{{}} @{} *{5}{c<{{}} @{}} l l}
			(0, & \multicolumn{2}{@{}c@{}}{\LeaderA,}   & 0,\overset{\sigma}{\pm 1},0, & \multicolumn{2}{@{}c@{}}{\LeaderB}      & {},0) & \text{if $\sigma \cdot w = \sigma$,} \\
			(0, & \dotsc, & \overset{\sigma}{x_{i,j}},  & 0,\dotsc,                 0, & \overset{\mathclap{\sigma \cdot w}}{1}, & \dotsc & {},0) & \text{if $\sigma \cdot w > \sigma$,} \\
			(0, & \dotsc, & \overset{\sigma\cdot w}{1}, & 0,\dotsc,                 0, & \overset{\mathclap{\sigma}}{0},         & \dotsc & {},0) & \text{if $\sigma \cdot w < \sigma$,}
		\end{array}
		\right.
	\end{equation}
	where the superscripts indicate the column of the respective entry.
	In other words, every row $\gamma(w)_{\sigma*}$ of $\gamma(w)$ contains at most two non-zero entries:
	either an entry $x_{i,j}$ in column $\sigma$ and an entry $1$ in column $\sigma \cdot w$ if $\sigma \leq \sigma \cdot w$,
	or a single non-zero entry $1$ in column $\sigma \cdot w$ in all other cases.
	If follows that $\ColMat(\gamma(w)^{\wedge S})$ has precisely the circuits $\Set{ \{\sigma, \sigma \cdot w\}; \sigma \in S \not\ni \sigma \cdot w}$.

	For \proofref{thm:partial shift of simple transpositions:2}, let $w = s_i$. Then $\frU(w) = \gamma(w)$,
	so \proofref{thm:partial shift of simple transpositions:2} follows from \proofref{thm:partial shift of simple transpositions:1}.
\end{proof}

\begin{remark}
	\Cref{thm:partial shift of simple transpositions} differs from a similar statement of Hibi--Murai~\cite{Murai+Hibi:2009}.
	Namely, in \cite[Lemma~2.5]{Murai+Hibi:2009} it is shown that a certain $t$-dependend construction $(-)_{ij}(t)$ of ideals in the exterior algebra
	satisfies $I_{\CShift_{(i\,j)}(S)} = (I_S)_{ij}(t)$ for $t = 0$
	and for $t\neq 0$ arises as $(I)_{ij}(t) = \lambda^t_{ij} \cdot I$ for some matrix $\lambda^t_{ij} \in \GL(n,\FF)$.
	In particular, the statement does not show how $\CShift_{(i\,j)}(-)$ is induced by an invertible $n\times n$-matrix,
	which is covered by \cref{thm:partial shift of simple transpositions}.

	Further, for $w = (i\ j)$, the matrix $\gamma(w)$ from \eqref{eq:cshift-matrix} differs from the matrix
	\[
		\frU(w) \;=\;
		\begin{pNiceMatrix}[first-col,first-row,small]
			  &   &        & i         &           &        &           & j &        &   \\
			  & 1 &        &           &           &        &           &   &        &   \\
			  &   & \ddots &           &           &        &           &   &        &   \\
			i &   &        & x_{ij}    & x_{i,i+1} & \cdots & x_{i,j-1} & 1 &        &   \\
			  &   &        & x_{i+1,j} & 1         &        &           &   &        &   \\
			  &   &        & \vdots    &           & \ddots &           &   &        &   \\
			  &   &        & x_{j-1,j} &           &        & 1         &   &        &   \\
			j &   &        & 1         &           &        &           & 0 &        &   \\
			  &   &        &           &           &        &           &   & \ddots &   \\
			  &   &        &           &           &        &           &   &        & 1
		\end{pNiceMatrix}\;.
	\]
	Both lie in the Bruhat cell of $w$; however, $\frU(w)$ is generic in this cell, but $\gamma(w)$ is not.
\end{remark}

\begin{proposition}
	A $k$-uniform hypergraph $S \subset \binom{[n]}{k}$ is shifted
	if and only if $\CShift_{s_i}(S) = S$ for all $i < n$.
\end{proposition}
\begin{proof}
	For $1 \leq i < j \leq n$, let $t_{ij}$ be the transposition $(i\ j)$.
	Assume that $S$ is not shifted, so there exists a $\sigma \in S$ and $i < j \in \sigma$
	such that $\sigma \cdot t_{ij} \notin S$.
	With $\sigma_k \coloneqq \sigma \cdot t_{kj}$ for $i \leq k < j$,
	there must be some $k$ s`ch that $\sigma_{k+1} \in S$ and $\sigma_{k+1} \cdot s_k = \sigma_k \notin S$.
	Then $\CShift_{s_k}(S) \ni \sigma_k$.
\end{proof}

\section{The partial shift graph}
\label{sec:partial-shift-graph}
In this section we investigate ways of relating the various partial shifts of all uniform hypergraphs of the same kind to each other.
As a first idea we could consider the directed graph where each $k$-uniform hypergraph on $n$ elements with $m$ hyperedges is a node, and each partial shift with respect any invertible matrix defines a directed edge.
Interestingly, that graph is not acyclic already for $k = m = 1$.
Recall that we embed $\SymmetricGroup{n}$ into $\GL(n, \EE)$, so partial shifting $\Delta_w$ by the permutation matrix $w$ makes sense,
but is different from the partial shift $\Delta_{\frR(w)}$ with respect to the permutation $w$. 
Let $S = \{\{i\}\}$ consists of a single $1$-hyperedge.
Then shifting by a permutation \emph{matrix} $w$ yields $\Delta_{w}(\{\{i\}\}) = \{\{i\cdot w\}\}$.
Therefore, we pass to a suitable subgraph, which will turn out to be acyclic;
see \cref{thm:partial shift graph acyclic}.
In this way, we arrive at a hierarchy of partial shifts that is compatible with the weak order on $\SymmetricGroup{n}$.

\begin{definition}
	For $n$, $k$ and $m \in \NN$,
	the \emph{partial shift graph $\ShiftGraph(n,k,m)$} is the directed graph
	with vertex set $\binom{\binom{[n]}{k}}{m}$ and edges $S \to T$
	if there is a $w \in \SymmetricGroup{n}$ such that $T = \Delta_{\frR(w)}(S)$.
\end{definition}
\begin{figure}[tb]
	\sbox0{\includegraphics{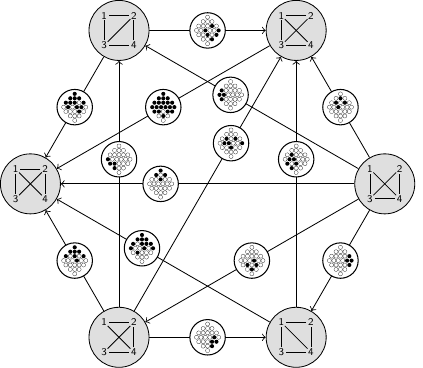}}%
	\raisebox{-\height}{\usebox0}%
	\hspace{1ex}%
	\raisebox{\abovecaptionskip-\height}{%
		\hspace{-\ExtraCaptionMargin}%
		\begin{minipage}{\linewidth-\wd0+2\ExtraCaptionMargin-1ex}%
			\caption{Partial shift graph $\ShiftGraph(4,2,5)$ of all graphs on four vertices with five edges.
				It contains precisely one shifted complex (i.e., one sink), which is the node on the left.
				For each edge $S \to T$ for $S, T \in \binom{[4]}{2}$,
				the filled circles in the small diagrams on that edge represents the set
				$\Set{w \in \SymmetricGroup{4}; \Delta_{\frR(w)}(S) = T}$,
				following the visualization of $\SymmetricGroup{4}$ from~\cites[Figure~2.4]{BjornerBrenti:2005}[Figure~3.65]{Stanley:EC1}.
				If a permutation $w$ is missing from all outgoing arrows from a complex $S$,
				this means that $\Delta_{\frR(w)}(S) = S$.
				In particular, $e$ is missing from all arrows,
				and $w_0$ is present in all arrows to the sink.
				\label{fig:PSG-4-2-5}
			}%
		\end{minipage}%
		\hspace{-\ExtraCaptionMargin}%
	}%
\end{figure}
\Cref{fig:PSG-4-2-5} shows a visualization of the partial shift graph $\ShiftGraph(4,2,5)$.
\Cref{thm:partial shift of simple transpositions:2} shows that the simple transpositions in $\SymmetricGroup{n}$
generate an acyclic subgraph of $\ShiftGraph(n,k,m)$.

\subsection*{Acyclicity of the partial shift graph}
Recall from the proof of \cref{thm:shift preserves cardinality}
the integer sequence $r(m)$ of a matrix $m$ with $r_0(m) \coloneqq 0$ and $r(m)_\ell \coloneqq \rk [(m_{*j})_{j \leq \ell}]$.
Analogously to \eqref{eq:rank-sequence-shift}, we obtain
\begin{equation}
	\label{eq:shift r-seq}
	\Delta_{\frR(w)}(S) = \Set[\big]{ \sigma \in \tbinom{[n]}{k}; r(\frR(w)^{\wedge S})_{\sigma-1} < r(\frR(w)^{\wedge S})_{\sigma}},
\end{equation}

We may compare rank sequences of matrix by the lexicographic order on $\NN^{\binom{n}{k}}$.
Our central result is that larger (in the weak order on $\SymmetricGroup{n}$) words
give rise to matrices $\frR(w)^{\wedge S}$ with lexicographically larger rank-sequences,
which in turn produce lexicographically smaller partial shifts.
From this, it follows easily that $\ShiftGraph(n,k,m)$ is acyclic; see \cref{thm:partial shift graph acyclic}.
Another remarkable consequence is that the full shift of a hypergraph is the lexicographically smallest one among all its partial shifts,
see \cref{thm:full shift lex-smallest}.

\begin{theorem}
	\label{thm:rank sequences lex order}
	Let $w \in \SymmetricGroup{n}$ and $\ell(ws_i) > \ell(w)$. Then
	\[
	r\bigl(\frR(w)^{\wedge S}\bigr) \ \lex\leq \ r\bigl(\frR(w s_i)^{\wedge S}\bigr) \enspace.
	\]
\end{theorem}
\begin{proof}
	\newcommand{\frW}{\mathfrak{w}}
	\newcommand{\frWS}{\mathfrak{ws}}
	\newcommand{\si}{\tilde{s}_i}
	We already know from \cref{thm:u(ws) and u(w)u(s)} that the compound matrices $\frR(w s_i)^{\wedge S}$ and $\frR(w)^{\wedge S} (\frR(s_i)^w)^{\wedge k}$ have the same column matroid.
	Therefore, with $\frW \coloneqq \frR(w)^{\wedge S}$
	and $\frWS \coloneqq \frR(w)^{\wedge S} (\frR(s_i)^w)^{\wedge k}$,
	it suffices to show $r(\frW) \lex\leq r(\frWS)$.
	The only $\FF$-transcendental entry of $\frR(s_i)^v$
	is $(\frR(s_i)^w)_{i,i} = x_{i \cdot w^{-1}, (i+1) \cdot w^{-1}}$,
	which we will abbreviate simply by $x$ in this proof.
	Enumerate $\binom{[n]}{k}$ lexicographically such that $\binom{[n]}{k} \eqqcolon \{\sigma_1 \lex< \dotsb \lex< \sigma_{\binom{n}{k}} \}$,
	and for $l \in \bigl[\binom{n}{k}\bigr]$,
	denote by $l \cdot \si$ the number such that $\sigma_{l \cdot \si} = \sigma_l \cdot s_i$.

	To avoid cluttered notation, we write $\frW_j$ and $\frWS_j$ for the $j$th column of $\frW$ and $\frWS$. 
	The columns of $\frWS$ satisfy
	\begin{equation}
		\label{eq:w-ws-transvection}
		\frWS_l = \pm \begin{cases*}
			\frW_l & if $l \cdot \tilde{s}_i = l$,\\
			\frW_{l \cdot \si} & if $l \cdot \tilde{s}_i < l$,\\
			\frW_{l \cdot \si} + x \frW_{l} & if $l \cdot \tilde{s}_i > l$,
		\end{cases*}
	\end{equation}
	where $x$ is transcendental over the entries of $\frW$.
	Note that \eqref{eq:w-ws-transvection} shows that $(\frR(s_i)^w)^k$ acts on the columns on $\frW$ by a transvection.

	We show inductively that $r(\frW) \lex\leq r(\frWS)$,
	which is the statement of the \lcnamecref{thm:rank sequences lex order}.
	If $\frW_1 = 0$, then $r_1(\frW) \leq r_1(\frWS)$ trivially.
	If $\frW_1 \neq 0$, then
	\begin{equation*}
		\frWS_1 = \pm\begin{cases*}
			\frW_1 & if $1 \cdot \si = 1$,\\
			x\frW_1 + \frW_{1 \cdot \si} & otherwise.
		\end{cases*}
	\end{equation*}
	In the first case, $\frWS_1$ is nonzero trivially.
	In the second, $\frWS_1$ is nonzero because $\frW_1$ is nonzero, $x$ is transcendental over $\FF$, and $\frW_1$ and $\frWS_{1 \cdot \si}$ have entries in $\FF$.
	Together, $r_1(\frW) \leq r_1(\frWS)$.

	By induction, assume for some $l$ that $r_k(\frW) = r_k(\frWS)$ for all $k < l$.
	We have to show that $r(\frW)_l \leq r(\frWS)_l$.
	If $r(\frW)_l = r(\frW)_{l-1}$, then $r(\frW)_l = r(\frW)_{l-1} \stackrel{\text{ass.}}{=} r(\frWS)_{l-1} \leq r(\frWS)_l$ trivially,
	and we are done.
	Hence, assume that $r(\frW)_l < r(\frW)_{l-1}$;
	in other words, $\frW_l \notin \Span{\frW_{<l}}$,
	and we have to show that $\frWS_l \notin \Span{\frWS_{<l}}$.
	In general, it follows from \eqref{eq:w-ws-transvection} that $\frW$ and $\frWS$ are of the form
	\begin{equation*}\fboxsep=0pt
		\begin{tikzcd}[row sep=large, column sep=2pt, arrows={out=-90, in=90, looseness=.5},
			X/.style = {dashed}, 
			co/.style = {crossing over},
			crossing over clearance=3pt,
			execute at end picture={
				\draw (A1.north west) rectangle (B1.south east);
				\draw (A2.north west) rectangle (B2.south east);
				\draw (C1.north west) rectangle (D1.south east);
				\draw (C2.north west) rectangle (D2.south east);
				\draw (E1.north west) rectangle (F1.south east);
				\draw (E2.north west) rectangle (F2.south east);
			}
			]
			\frW = \bigl(  & |[alias=A1]| {} \ar[d, X]\ar[ dr] & {} \ar[co, dl] & |[alias=B1]| {} \ar[drr] \ar[X, d] & \frW_l \ar[X, d]\ar[drrr]  & |[alias=C1]| {} \ar[co, dll] & |[alias=D1]| {} \ar[X, d] \ar[drr] & \frW_{l \cdot \si} \ar[co, dlll] & |[alias=E1]| {} \ar[dll, co] & {} \ar[X, d]\ar[dr] & |[alias=F1]| {} \ar[dl, co] & \bigr)             \\
			\frWS = \bigl( & |[alias=A2]| {}                   & {}             & |[alias=B2]| {}                    & X\frW_l+\frW_{l \cdot \si} & |[alias=C2]| {}              & |[alias=D2]| {}                    & \frW_{l}                         & |[alias=E2]| {}              & {}                  & |[alias=F2]| {}             & \bigr)\mathrlap{,}
		\end{tikzcd}
	\end{equation*}
	where the rectangles stand for blocks of columns, the black arrows correspond to a copy of a column, and the dashed arrows correspond to a copy of a column multiplied by $x$.

	Without loss of generality, we may assume (by deleting columns) that all columns of the submatrix $\frW_{<l}$ are linearly independent.
	Then, by the assumption $r(\frW_{<l}) = r(\frWS_{<l})$, all columns of $\frWS_{<l}$ are linearly independent, too.
	Note that by assumption that $\frW_l \notin \Span{\frW_{\leq l}}$, the column $\frW_l$ is not among the deleted ones, and all columns of $\frW_{\leq l}$ are linearly independent.
	By deleting rows of $\frW$, we may assume that $\frW_{\leq l}$ is invertible,
	and we have to show that $\frWS_{\leq l}$ is invertible.
	Let
	\begin{align*}
		I &\coloneqq \Set{j \leq l; j\cdot\si \leq l},&
		J &\coloneqq \Set{j \leq l; j\cdot\si > l},
	\end{align*}
	and denote by $\frW_I$ and $\frWS_I$ the submatrices of $\frW$ and $\frWS$ with the respective column indices.
	Then $\frWS_{I\cdot\si} = \frW_I t$ for a transvection $t$, and $\frWS_J = x\frW_J + \frW_{J\cdot\si}$, where $j > l $ for all $J\cdot\si$.
	As a transvection, $t$ has determinant one, so we obtain
	\[
	\det\frWS_{\leq l} = \pm\det(\frW_I t, x\frW_J+\frW_{J\cdot\si}) = \pm\det(\frW_I, x\frW_J+\frW_{J\cdot\si})
	\]
	Elements $x_{11},\dotsc,x_{nn} \in \EE$ are algebraically independent over $\FF$
	if and only if the evaluation map $\FF[X_{11},\dotsc,X_{nn}] \to \EE$, $X_{ij} \mapsto x_{ij}$ is injective.
	Therefore, $\det\frWS_{\leq l}$ is nonzero if and only if the polynomial
	\[
	\det(\frW_I, X\frW_J+\frW_{J\cdot\si}) = \pm\underbrace{\det(\frW_I, \frW_J)}_{\pm\det\frW_{\leq l}} X^{\abs{J}} + \mathcal{O}(X^{\abs{J}-1}).
	\]
	in $\FF[X_{11},\dotsc,X_{nn}]$ is nonzero,
	which is the case because the leading coefficient of this polynomial is nonzero by assumption.
	Therefore, $\det\frWS_{\leq l} \neq 0$, which proves the \lcnamecref{thm:rank sequences lex order}.
\end{proof}

Recall that we endow $\binom{[n]}{k}$ with the total order $\lex\leq$.
This defines the lexicographic order on $\binom{\binom{[n]}{k}}{m}$ for any $m$.
As a corollary of \cref{thm:rank sequences lex order}, we obtain:

\begin{corollary}
	\label{thm:partial shift graph acyclic}
	If $\ell(ws_i) > \ell(w)$, then $\Delta_{\frR(w)}(S) \lex\geq \Delta_{\frR(ws_i)}(S)$.
	In particular, the partial shift graph is acyclic.
\end{corollary}
\begin{proof}
	For a matrix $m$, let $R(m) \coloneqq \Set{j; r(m)_j > r(m)_{j-1}}$.
	It is easy to see that for two matrices $m$ and $n$ of the same rank, we have $R(m) \lex\leq R(n)$
	if and only if $r(m) \lex\geq r(n)$.
	Now, note that $\Delta_{\frR(w)}(S) = R(m^{\wedge S})$ by \eqref{eq:shift r-seq}.
	It follows from \cref{thm:rank sequences lex order} that if $\ell(ws_i) > \ell(w)$, then
	$\Delta_{\frR(w)}(S) \lex\geq \Delta_{\frR(ws_i)}(S)$.
\end{proof}

\begin{remark}
	\label{rmk:sinks-shifted}
	If $w = s_{i_1} \dotsm s_{i_\ell}$, then
	\[
	S = \Delta_{\frR(e)}(S) \lex\geq \Delta_{\frR(s_{i_1})}(S) \lex\geq \dotsb \lex\geq \Delta_{\frR(w)}(S) \lex\geq \dotsb \lex\geq \Delta_{\frR(w_0)}(S) = \Delta(S).
	\]
	In particular, $S$ is a sink of $\ShiftGraph(n,k,m)$ if and only if $S$ is shifted.
\end{remark}

\pagebreak
\begin{example}
	Let $S = \bigl\{ \{1,2\}, \{1,4\}, \{2,3\} \bigr\} \subseteq \binom{[4]}{2}$
	and $w = s_2 s_3 s_2$.
	Then
	\begin{align*}
		\frR(w)^{\wedge S} &= \begin{psmallmatrix}
			x_{24} & x_{23} & 1 & 0                      & 0       & 0  \\
			1      & 0      & 0 & 0                      & 0       & 0  \\
			0      & 0      & 0 & \substack{-x_{23}x_{34} \\+ x_{24}} & -x_{34} & -1
		\end{psmallmatrix},&
		\frR(w s_1)^{\wedge S} &= \begin{psmallmatrix}
			-x_{24} & x_{14}x_{23}           & x_{14}  & x_{23} & 1 & 0  \\
			-1      & 0                      & 0       & 0      & 0 & 0  \\
			0       & \substack{-x_{23}x_{34} \\+ x_{24}} & -x_{34} & 0      & 0 & -1
		\end{psmallmatrix},
	\end{align*}
	which have rank sequences
	\begin{align*}
		r(\frR(w)^{\wedge S}) &= (1,2,2,3,3,3) & &\lex< &
		r(\frR(w s_1)^{\wedge S}) &= (1,2,3,3,3,3)\\
		\intertext{and give rise to the partial shift complexes}
		\Delta_{\frR(w)}(S) &= \bigl\{ \{1,2\}, \{1,3\}, \{2,3\} \bigr\} & &\lex>&
		\Delta_{\frR(w s_1)}(S) &= \bigl\{ \{1,2\}, \{1,3\}, \{1,4\} \bigr\}.
	\end{align*}
\end{example}

Another consequence of \cref{thm:partial shift graph acyclic} is that the full shift of a hypergraph
is the smallest partial shift of it that is shifted:

\begin{corollary}\label{thm:full shift lex-smallest}
	For every hypergraph $S$ on $n$ vertices,
	the full shift $\Delta(S)$ is the lexicographically smallest element of $\Set{\Delta_{\frR(w)}(S); w \in \SymmetricGroup{n}}$.%
\end{corollary}

For an example, see the partial shifts of the six-vertex triangulation of $\RR\PP^2_k$ in \cref{tab:shifted partial shifts of RP2}.

\subsection*{The contracted partial shift graph}
By Theorem \ref{thm:partial shift graph acyclic} the partial shift graph $\ShiftGraph(n,k,m)$ is a finite acyclic graph,
so it has at least one sink; possibly more.
As mentioned in \cref{rmk:sinks-shifted}, the sinks of $\ShiftGraph(n,k,m)$ are precisely the shifted complexes.
We define the following graph, whose nodes are in bijection with the shifted complexes in $\ShiftGraph(n,k,m)$.
By the quotient of a (directed) graph $G = (V, E)$ by an equivalence relation on the vertex set $V$
we mean the directed graph $G/{\sim} = (\bar{V}, \bar{G})$ with $\bar{V} = V/{\sim}$ as set of equivalence classes,
and (directed) edges $\bar{E} = \Set{([v], [w]); (v,w) \in E}$.
Note that a quotient of a directed acyclic graph need not be acyclic in general.

\begin{definition}
	\label{def:contracted-shift-graph}
	The \emph{contracted partial shift graph $\ContractedGraph(n,k,m)$}
	is the directed quotient graph $\ShiftGraph(n,k,m)/{\sim}$
	for $S \sim T$ if $\Delta(S) = \Delta(T)$.
\end{definition}

In other words,the nodes of $\ContractedGraph(n,k,m)$ are in bijection to shifted complexes in $\binom{\binom{[n]}{k}}{m}$,
and there is an edge $S \to T$ in $\ContractedGraph(n,k,m)$ between nodes represented by the shifted complexes $S$ and $T$
if there are complexes $S'$, $T'$ such that $\Delta(S') = S$, $\Delta(T') = T$ and $\Delta_w(S) = T$ for some $w$.
The graph $\ContractedGraph(n,k,m)$ can equivalently be seen as a contraction of $\ShiftGraph(n,k,m)$ along all edges $S \to \Delta(S)$.
Observe that the edges of the contracted partial shift graph necessarily correspond to partial shifts $S \mapsto \Delta_{\frR(w)}(S)$ with $\Delta_{\frR(w)}(S) \neq \Delta(S)$.

For a $S \in \ShiftGraph(n,k,m)$, we also define the \emph{contracted shift graph $\ContractedGraph(S)$ induced by $S$} as follows.
Let $\ShiftGraph(S)$ be the subgraph if $\ShiftGraph(n,k,m)$ induced by the nodes $T$ for which there exists a path from $S$ to $T$ (including $S$).
Then $\ContractedGraph(S) \coloneqq \ShiftGraph(S)/{\sim}$, for the equivalence relation $\sim$ from \cref{def:contracted-shift-graph}.
Of course, the entire construction still depends on the field $\FF$.
\Cref{exmp:rp26:contracted} below shows that different choices for $\FF$ (or rather its characteristic) may yield different contracted partial shift graphs.

\section{Partial Shifting to Near Cones}
\label{sec:simplicial}
We now study partial shifts of simplicial complexes.

\subsection*{Simplicial complexes}
Let $K$ be a simplicial complex on the vertex set $[n]$.
That is, $K$ is a subset of $2^{[n]}$ which is closed with respect to taking subsets.
An $s$-face of $K$ is an $(s+1)$-element subset of $K$.
The \emph{$s$-faces} of $K$ form a uniform hypergraph, which we write $K^s$.
The \emph{dimension} of $K$ is the maximal number $s$ for which $K$ has an $s$-face.
The relevance of (partial) shifting to simplicial complexes comes from the following observation.

\begin{proposition}
	\label{thm:shifting preserves complexes}
	Let $K$ be a simplicial complex on the vertex set $[n]$, and let $g\in\GL(n,\EE)$.
	Then the uniform hypergraphs $\Delta_g(K^s)$, for $0\leq s\leq\dim K$, form a simplicial complex, denoted $\Delta_g(K)$.
        Moreover, the f-vectors of $K$ and $\Delta_g(K)$ agree.
\end{proposition}
\begin{proof}
	Scrutinizing the proof for the full shift from~\cite[288]{BjornerKalai:1988} reveals that the standing assumption that the matrix $g$ is generic is never used.
	\Cref{thm:shift preserves cardinality} readily implies that $K$ and $\Delta_g(K)$ share the same f-vector.
\end{proof}

In~\cite[Theorem~3.1]{BjornerKalai:1988} it was shown that the full shift preserves the Betti numbers.
Our next goal is to prove a far reaching generalization of this statement.
Yet the following example shows that arbitrary partial shifts may change the Betti numbers.

\begin{example}\label{exmp:rp26:contracted}
	We consider the triangulation $K=\RR\PP_6^2$ of the real projective plane with six vertices.
	For $\FF=\QQ$, the four shifted complexes corresponding to the nodes of $\ContractedGraph(K)$ and their Betti numbers $\beta_i(-)$ are listed in \cref{tab:shifted partial shifts of RP2}.
	Note that shifted complexes are homotopy equivalent to a wedge of spheres, so their Betti numbers do not depend on the coefficients
	(but the assignment $K \mapsto \Delta_{\frR(w)}(K)$ of course does).
	Since shifted complexes are near cones (see \cref{def:near cone}), their Betti numbers can be calculated using the formula from \cref{thm:betti near cone} below.

	\begin{table}[t]
			\centering
			\caption{The six vertex triangulation $K$ of $\RR\PP^2_k$ and the four shifted complexes in $\ContractedGraph(\RR\PP^2_6)$, together with their rational Betti numbers,
				where $w_1 = w_0s_1$, $w_2 = w_0s_4s_3s_2s_1$, $w_3 = w_0s_4s_2s_3s_2s_1$. See \cref{exmp:rp26:contracted}.
				\label{tab:shifted partial shifts of RP2}}
			$\begin{array}{r @{} >{{}} l <{{}} @{} l l c }
				\toprule
				    \multicolumn{3}{c}{\text{complex}}     & \multicolumn{1}{c}{\text{facets}}                                 & \multicolumn{1}{r}{\mathllap{\text{Betti numbers}}} \\ \midrule
				K & \coloneqq & \RR\PP^2_6                 & \{ 125, 126, 134, 135, 146, 234, 236, 245, 356, 456 \}            & (1, 0, 0)                                           \\
				A & \coloneqq & \Delta^\QQ(K)              & \{ 123, 124, 125, 126, 134, 135, 136, 145, 146, 156\}             & (1,0,0)                                             \\
				B & \coloneqq & \Delta^\QQ_{\frR(w_1)}(K) & \{ 123, 124, 125, 126, 134, 135, 136, 145, 146, 234, 56\}         & (1,1,1)                                             \\
				C & \coloneqq & \Delta^\QQ_{\frR(w_2)}(K) & \{ 123, 124, 125, 126, 134, 135, 136, 234, 235, 236, 45, 46, 56\} & (1,3,3)                                             \\
				D & \coloneqq & \Delta^\QQ_{\frR(w_3)}(K) & \{ 123, 124, 125, 126, 134, 135, 145, 234, 235, 245, 36, 46, 56\} & (1,3,3)                                             \\ \bottomrule
			\end{array}$
	\end{table}

	The integral homology of $\RR\PP^2$ reads $H_0=\ZZ$, $H_1=\ZZ/2\ZZ$, and $H_2=0$.
	Hence the rational Betti numbers of $\RR\PP^2_6$ read $(1, 0, 0)$, which agree with the ones of $A$,
	while the $\GF{2}$-Betti numbers of $\RR\PP^2_6$ read $(1, 1, 1)$, which agree with the ones of $B$.
	Indeed, while we obtain $A = \Delta^\QQ(K)$ as the full shift in characteristic zero, we obtain $B = \Delta^{\GF{2}}(K)$ as the the full shift in characteristic two.
	The contracted partial shift graph in characteristic zero has the four nodes $A,B,C,D$, while $A$ is missing in characteristic two; see \cref{fig:rp26:contracted}.
	\begin{figure}[t]
		\centering
			\tikzset{x=5mm, y=5mm, every edge/.append style={->}, every node/.style={shape=circle, inner sep=1pt, outer sep=1pt, fill}}
			\begin{tikzpicture}
				\node[label=left:$D$]  (D) at (-2, 0) {};
				\node[label=right:$C$] (C) at ( 2, 0) {};
				\node[label=above:$A$] (A) at ( 0, 1) {};
				\node[label=below:$B$] (B) at ( 0,-1) {};
				\draw (A) edge (B) edge (C) edge (D)
				(B) edge (C) edge (D);
			\end{tikzpicture}
			\qquad\qquad
			\begin{tikzpicture}
				\node[label=below:$B$] (B) at ( 0,-1) {};
				\node[label=left:$D$]  (D) at (-2, 0) {};
				\node[label=right:$C$] (C) at ( 2, 0) {};
				\draw (B) edge (C) edge (D);
			\end{tikzpicture}
			\caption{Contracted partial shift graphs $\ContractedGraph(\RR\PP^2_6)$ with respect to $\QQ$ (left) and $\GF{2}$ (right).
				The complexes $A$, $B$, $C$ and $D$ are explained in \cref{tab:shifted partial shifts of RP2}.
				\label{fig:rp26:contracted}
			}
	\end{figure}
\end{example}

\subsection*{Near cones}
We now turn to a class of simplicial complexes which is particularly interesting for algebraic shifting~\cites[\S4]{BjornerKalai:1988}{Nevo:2005}.
In order to write down the definition we let $\RepVert{\sigma}{j}{i} \coloneqq \sigma \setminus \{j\} \cup \{i\}$, where $\sigma \subseteq [n]$, $i\in[n]\setminus\sigma$ and $j\in\sigma$ with $i<j$.
\begin{definition}
	\label{def:near cone}
	A simplicial complex $K$ is a \emph{near cone} if for every $\sigma \in K$ with $1 \notin \sigma$ we have $\RepVert{\sigma}{i}{1} \in K$
	for each $i \in \sigma$.
\end{definition}
Every shifted complex is a near cone. The next result gives an elementary description of the homology of any near cone.
That observation is crucial for seeing that algebraic shifting preserves the Betti numbers~\cite[Theorem~3.1]{BjornerKalai:1988}.
We use $\beta_k(K)$ to denote the $k$th Betti number (with respect to $\FF$) of a simplicial complex $K$.

\begin{proposition}[{\cite[Theorem~4.3]{BjornerKalai:1988}}]
	\label{thm:betti near cone}
	Let $K$ be a near cone.
	Then $K$ is homotopy equivalent to a wedge of spheres and $\beta_k(K) = \abs{\Set{\sigma \in K^k ; \sigma \cup \{1\} \notin K}}$.
\end{proposition}

Note that we can write the combinatorial shift of $K$ with respect to the transposition $(i \ j)$ with $i<j$ as
\[
\CShift_{(i\ j)}(\sigma, K) \ = \ \begin{cases}
	\RepVert{\sigma}{j}{i} & \text{if $i \notin \sigma \ni j$ and $\RepVert{\sigma}{j}{i} \notin K$,}\\
	\sigma & \text{otherwise.}
\end{cases}
\]
An immediate consequence from the definition that, for any near cone $K$, we have $\CShift_t(K) = K$ for any transposition $t = (1 \enspace i)$.
This observation has the following generalization.
\begin{proposition}
	Let $K$ be a near cone on $n$ vertices, and let $t \in \SymmetricGroup{n}$ be an arbitrary transposition.
	Then $\CShift_t(K)$ is again a near cone.
\end{proposition}
\begin{proof}
	By the preceeding discussion it is enough to consider $t = (i \enspace j)$ where $1 < i < j$.
	For every $\sigma \in K$ with $1 \notin \sigma$ and every $r \in \sigma$, we have to show that there is a $\tau \in K$ such that $\RepVert{\CShift_t(\sigma, K)}{r}{1} = \CShift_t(\tau, K)$.

		If $i, j \in \sigma$ and $r = i$, then $\tau = \RepVert{\sigma}{r}{1}$:
		we have \[\CShift_t(\RepVert{\sigma}{r}{1}) = \begin{cases}
			\RepVert{\RepVert{\sigma}{i}{1}}{j}{i}& \text{if $\RepVert{\RepVert{\sigma}{i}{1}}{j}{i} \notin K$}\\
			\RepVert{\sigma}{i}{1} &\text{otherwise}
		\end{cases}.\]
		The first case can not occur because $\RepVert{\RepVert{\sigma}{i}{1}}{j}{i} = \RepVert{\sigma}{j}{1} \in K$ since $K$ is a near cone,
		so $\CShift_t(\RepVert{\sigma}{r}{1}) = \RepVert{\sigma}{i}{1}$.
		Now $i \in \sigma$ implies $\sigma = \CShift_t(\sigma, K)$ and thus $\RepVert{\sigma}{r}{1} = \RepVert{\CShift_t(\sigma, K)}{r}{1}$.

		If $i \notin \sigma$ and $r = j$, then $\tau = \sigma$:
		we have \[\RepVert{\CShift_t(\sigma, K)}{j}{1} = \begin{cases}
			\RepVert{\RepVert{\sigma}{j}{i}}{j}{1} & \text{if $\RepVert{\sigma}{j}{i} \notin K$}\\
			\RepVert{\sigma}{j}{1} & \text{otherwise}
		\end{cases}.\]
		In the first case, $\RepVert{\RepVert{\sigma}{j}{i}}{j}{1} = \RepVert{\sigma}{j}{i} = \CShift_t(\sigma, K)$
		because $j \notin \RepVert{\sigma}{j}{i}$.
		In the second case, $\RepVert{\sigma}{j}{1} = \CShift_t(\RepVert{\sigma}{j}{1}, K)$
		again because $j \notin \RepVert{\sigma}{j}{i}$.

		In all other cases, one sees easily that $\tau = \RepVert{\sigma}{r}{1}$.
\end{proof}

The following result has already been suggested in~\cite[Remark~3.2]{BjornerKalai:1988}.
In loc.\ cit.\ it had not been worked out in detail, probably because it is of no direct use for full shifting.
However, for partial shifting the statement will turn out to be crucial.
Our proof is a variation of the \emph{Permutation Lemma} in~\cite[288]{BjornerKalai:1988}:

\begin{lemma}
	\label{thm:col-cone}
	Let $K$ be a simplicial complex on $n$ vertices.
	Further, let $g\in\GL(n,\EE)$ be a matrix whose entries in the first column are algebraically independent over $\FF$ and the other entries of $g$.
	Then $\Delta_g(K)$ is a near cone.
\end{lemma}
\begin{proof}
	\newcommand{\sigmaA}{\RepVert{\sigma}{1}{j}}
	\newcommand{\tauA}{\RepVert{\tau}{1}{j}}
	We assume that the field $\EE=\FF(g_{11},\dots,g_{n1})$ is generated by the first column of $g$ is a purely transcendental extension of $\FF$.
	Moreover, we assume that all other coefficients of $g$ lie in $\FF$.
	That is, $\EE=\FF(g)$ is isomorphic to the field of $\FF$-rational functions in $n$ indeterminates.
	To show that $\Delta_g(K)$ is a near cone,
	let $\sigma \subseteq [n]$ be a nonface of $\Delta_g(K)$ of dimension $k$ with $1 \in \sigma$.
	We have to show that for every $j \notin \sigma$, also $\sigmaA$ is a nonface of $\Delta_g(K)$.

	Let $S=K^k$ be the set of $k$-faces of $K$.
	By \cref{def:partial-g-shift}, $\sigma$ is a nonface of $\Delta_g(K)$ if there are coefficients $\gamma_\tau \in \EE$ parameterized by $\tau \lex< \sigma$ such that for all $k$-faces $\rho$ of $K$, we have
	\begin{equation}
		\label{eq:col-cone:lin rel}
		g^{\wedge S}_{\rho\sigma} - \smashoperator{\sum_{\tau \lex< \sigma}} \gamma_\tau g^{\wedge S}_{\rho\tau} \ = \ 0 \enspace .
	\end{equation}
	Observe that $1 \in \sigma$ together with $\tau \lex< \sigma$ implies that $1 \in \tau$.
	Since $\EE = \FF(g_{11},\dotsc,g_{n1})$ we may treat \eqref{eq:col-cone:lin rel} as a collection of rational equations
	\begin{equation}
		\label{eq:col-cone:rat rel}
		r_\sigma(g_{11},\dotsc,g_{n1}) - \smashoperator[l]{\sum_{\tau \lex< \sigma}} \frac{q_\tau(g_{11},\dotsc,g_{n1})}{p(g_{11},\dotsc,g_{n1})} r_\tau(g_{11},\dotsc,g_{n1}) \ = \ 0
	\end{equation}
	for polynomials $p, q_\tau, r_\sigma, r_\tau$ in the $n$ indeterminates $g_{11},\dotsc,g_{n1}$ with coefficients in $\FF$.
	Here $p$ and $q_\tau$ arise from bringing the rational functions $\gamma_\tau$ to a common denominator;
	and $r_\sigma$ and $r_\tau$, respectively, are the $(\rho, \sigma)$- and $(\rho,\tau)$-minors of the matrix $g$.
	Let $j \in [n] \setminus \sigma$.
	To show that $\sigmaA \notin K$, we will show that evaluating \eqref{eq:col-cone:rat rel}
	on $h_{11},\dotsc, h_{n1}$ yields zero, where we let $h_{i1} \coloneqq g_{i1} + g_{ij}$.
	Multiplying both sides of \eqref{eq:col-cone:rat rel} with the polynomial $p$ yields that $z \coloneqq	pr_\sigma - \sum_{\tau \lex< \sigma} q_\tau r_\tau$
	is the zero polynomial in $\FF[g_{11},\dotsc,g_{n1}]$.
	In particular, evaluating $z$ on $h_{11},\dotsc, h_{n1}$ yields zero.
	The substitution of $g_{i1}$ by $h_{i1}$ is affine-linear over $\FF$.
	Consequently, $\EE=\FF(g_{11},\dotsc,g_{n1})=\FF(h_{11},\dotsc, h_{n1})$, and $h_{11},\dotsc, h_{n1}$ are $\FF$-algebraically independent.
	By construction the denominator polynomial $p$ is nonzero, and hence $p(h_{11},\dotsc,h_{n1}) \neq 0$.
	We obtain
	\begin{equation}
		\label{eq:col-cone:rat rel2}
		r_\sigma(h_{11},\dotsc,h_{n1}) - \smash{\smashoperator[l]{\sum_{\tau \lex< \sigma}} \frac{q_\tau(h_{11},\dotsc,h_{n1})}{p(h_{11},\dotsc,h_{n1})}} r_\tau(h_{11},\dotsc,h_{n1}) \ = \ 0 \enspace,
	\end{equation}
	where\leavevmode\vspace{\abovedisplayshortskip-\abovedisplayskip}
	\begin{align*}
		r_\sigma(h_{11},\dotsc,h_{n1}) \ &= \ r_\sigma(g_{11},\dotsc,g_{n1}) + r_{\sigma}(g_{1j},\dotsc,g_{nj}) \ = \ g^{\wedge S}_{\rho\sigma} \pm g^{\wedge S}_{\rho\sigmaA}  \quad \text{and} \\[3pt]
		r_\tau(h_{11},\dotsc,h_{n1}) \ &= \ \begin{cases}
			g^{\wedge S}_{\rho\tau} & \text{if $j \in \tau$,} \\
			\smash[b]{g^{\wedge S}_{\rho\tau} \pm g^{\wedge S}_{\rho\tauA}} & \text{if $j \notin \tau$}
		\end{cases}
	\end{align*}
	because the polynomials $r_\sigma$ and $r_\tau$ arise as minors of $g^{\wedge S}$.
	Since $\sigma \lex< \sigmaA$ the $k$-faces $\tau$ in \eqref{eq:col-cone:rat rel2} also satisfy $\tau \lex\leq \sigmaA$.
	Clearly $\tau \lex< \sigma$ also implies $\tauA \lex< \sigmaA$.
	Hence, by \eqref{eq:col-cone:rat rel2} we have $g^{\wedge S}_{*\sigmaA}  \in \Span[\big]{g^{\wedge S}_{*\upsilon}; \upsilon \lex< \sigmaA} $.
	This establishes that $\sigmaA $ is a nonface of $\Delta_g(K)$.
\end{proof}
\begin{remark}
	\label{rmk:first column independent}
	In the situation of \cref{thm:col-cone}, it suffices to assume that the entries $g_{11},\dotsc,g_{n-1,1}$
	are algebraically independent over $\FF(g_{n1},g_{12},\dotsc,g_{nn})$, and that $g_{n,1}$ is not zero.
	In this case, multiplying the first column of $g$ by a transcendental $y$ over $\FF(g)$  yields a matrix
	that has algebraically independent entries in the first column without changing $\Delta_g(-)$.
\end{remark}
\pagebreak
\begin{example}
	The $n$-cycle $c_n \coloneqq (1\ 2\ \ldots\ n)$ gives rise to the matrix
	\[
	\frR(c_n) = \begin{psmallmatrix}
		x_{1,n}   & 1 &        &   \\
		\vdots    &   & \ddots &   \\
		x_{n-1,n} &   &        & 1 \\
		1         & 0 & \cdots & 0
	\end{psmallmatrix}.
	\]
	The six-vertex triangulation $K$ of the real projective plane from \cref{tab:shifted partial shifts of RP2} is not a near cone.
	In contrast, its partial shift $\Delta_{\frR(c_6)}(K)$ (with coefficients in $\GF{2})$ has facets $\{123$, $124$, $125$, $126$, $134$, $135$, $136$, $146$, $156$, $236$, $45\}$ and thus is a near cone, but not a shifted complex.
\end{example}

We would like to find a statement that relates $w \in \SymmetricGroup{n}$ to the topology of $\Delta_{\frR(w)}(K)$.
Putting together statements involved in the proof of~\cite[Theorem~3.1]{BjornerKalai:1988} we have:

\begin{lemma}
	\label{thm:bjorner kalai collected}
	For any simplicial complex $K$ we have
	\[\beta_k(K) \ = \ \abs{K^k} - \abs{\Set{\sigma \in \Delta(K^{k+1}); 1 \in \sigma}} - \abs{\Set{\sigma \in \Delta(K^{k}); 1 \in \sigma}} \enspace.\]
\end{lemma}

We now have enough to prove the main theorem of this section.
It relates a permutation $w \in \SymmetricGroup{n}$ to the topology of of the partial shift $\Delta_{\frR(w)}(K)$ with respect to $w$.
\begin{theorem}
	\label{thm:n-cycle}
	Let $c_n$ be the $n$-cycle $(1 \ 2 \ 3 \ \ldots \ n) \in \SymmetricGroup{n}$ and $v, w \in \SymmetricGroup{n}$ such that $w = c_nv$ and $\ell(c_n v) = \ell(c_n)+\ell(v)$.
	Then for any simplicial complex $K$,
	\begin{enumerate}
		\item $\Delta_{\frR(w)}(K)$ is a near cone, and\label{thm:n-cycle:i}
		\item $\Delta_{\frR(w)}(K)$ has the same Betti numbers as $K$.\label{thm:n-cycle:ii}
	\end{enumerate}
\end{theorem}

\begin{proof}
		For \proofref{thm:n-cycle:i},
		we have $\inv c_n = \Set{(i, n); i < n}$, and by \cref{thm:inv(vw)}, we have that $\inv c_n \subseteq \inv c_nv$.
		Therefore, by definition, the first $n-1$ entries of the last column of the matrix $\frU(c_nv) = (e + \frX|_{\inv c_nv})$ are algebraically independent.
		Hence, the first $n-1$ entries in the first column of $\frR(c_n)$ are algebraically independent.
		Also by \cref{thm:inv(vw)} we have $\inv c_n \cap (\inv v) \cdot c_n^{-1} = \emptyset$, which implies that $1 \cdot v = 1$.
		Since $v$ fixes $1$, the matrix $\frR(c_nv)$ has the same first column as $\frR(c_n)$.
		The first statement then follows from \cref{thm:col-cone}, together with \cref{rmk:first column independent}.

		\medskip
		To prove \proofref{thm:n-cycle:ii} it is enough to show:
		\begin{claim}
			\label{claim:card 1 in sigma}
			$
			\abs{\Set{\sigma \in \Delta(K^k); 1 \in \sigma}} = \abs{\Set{\sigma \in \Delta_{\frR(w)}(K^k); 1 \in \sigma}}.
			$
		\end{claim}
		Before proving the claim, we first note the following.
		For a matrix $m$, let $I_m$ denote the set of column indices corresponding to a lexicographically minimal basis of $\ColMat(m)$.
		\begin{claim}
			\label{claim:I(mw)}
			Let $m$ be a matrix, $J$ be an initial set of column indices of $m$,
			and $w$ be a permutation of the columns of $m$ such that $J\cdot w = J$.
			Then $\abs{I_{m} \cap J} = \abs{I_{mw}\cap J}$.
		\end{claim}
		\begin{claimproof}[Proof of \cref{claim:card 1 in sigma}]
			Let $S \coloneqq K^{k}$ be the $k$-simplices of $K$.
			By \cref{thm:u(ws) and u(w)u(s)},
			\begin{equation}
				\label{eq:proof near cones:column matroids}
				\ColMat(\frR(w_0)^{\wedge S})
				\;=\;
				\ColMat\bigl((\frR(w) \: \frR(w^{-1}w_0)^{w})^{\wedge S}\bigr).
			\end{equation}
			By \cref{thm:inv(vw)}, we have that $\inv c_n \cap (\inv c_n^{-1}w_0) \cdot c_n^{-1} = \emptyset$, which implies $1 \cdot c_n^{-1}w_0 = 1$, hence $w^{-1}w_0 = v^{-1}c_n^{-1}w_0$ must also fix $1$.
			Extending the action on vertices to an action on sets we see that for $J \coloneqq \Set{\sigma \in \binom{[n]}{k}; 1 \in \sigma}$ we have $J \cdot w^{-1}w_0 = J $.
			Note that the thus defined set $J$ is an initial segment of $\binom{[n]}{k}$ with respect to the lexicographic order,
			so $J$ satisfies the requirements of \cref{claim:I(mw)}.
			This also implies that $(w_0w)^{\wedge k}$ permutes the columns of $\bigl( \frR(w)\frR(w^{-1}w_0)^w \bigr)^{\wedge S}$, possibly changing signs of the columns.
			For the same reason, $(w_0 w)^{\wedge k}$ permutes the columns of $\frR(w_0)^{\wedge S}$ indexed by $J$;
			hence, we have that columns indexed by $J$ in $\frR(w_0)^{\wedge S}$ and $\frR(w_0)^{\wedge S}(w_0w)^{\wedge k}$ span the same vector space over $\EE$.
			Let $I_{\frR(w_0)}$ be the subset of $\binom{[n]}{k}$ that indexes the lexicographically minimal bases for the column span of $\frR(w_0)^{\wedge S}$,
			and define $I_{\frR(w_0)w_0w}$ and $I_{\frR(w)\frU(w^{-1}w_0)^{w}}$ analogously.
			We have
			\begin{align*}
				&\phantom{{}={}} \abs{\Set{\sigma \in \Delta(K^k); 1 \in \sigma}}\\
				&= \abs{I_{\frR(w_0)} \cap J} &&\text{by definition of $\Delta(K^k)$} \\
				&= \abs{I_{\frR(w_0)w_0w} \cap J} &&\text{by the above reasoning and \cref{claim:I(mw)}}\\
				&= \abs{I_{\frR(w)\frR(w^{-1}w_0)^{w}} \cap J} &&\text{by \eqref{eq:proof near cones:column matroids}}\\
				&= \abs{I_{\frR(w)\frU(w^{-1}w_0)^{w}} \cap J} &&\text{again by \cref{claim:I(mw)}}\\
				&= \abs{I_{\frR(w)} \cap J} &&\text{by \cref{thm:invariance upper triangular} since $\frU(w^{-1}w_0)^{w}$ is upper triangular}\\
				&= \abs{\Set{\sigma \in \Delta_{\frR(w)}(K^k); 1 \in \sigma}} &&\text{by definition of $\Delta_{\frR(w)}(K^k)$.}
			\end{align*}
			This establishes the claim.
		\end{claimproof}
		By \cref{thm:bjorner kalai collected} and by the fact that shifting preserves the f-vector, we have
		\begin{align*}
			\beta_k(K) \ &=\  \abs{\Delta_{\frR(w)}(K^{k})} - \abs{\Set{\sigma \in \Delta_{\frR(w)}(K^{k+1}); 1 \in \sigma}} - \abs{\Set{\sigma \in \Delta_{\frR(w)}(K^{k}); 1 \in \sigma}} \\
			&=\ \abs{\Set{\sigma \in \Delta_{\frR(w)}(K^k); \{1\} \cup \sigma \notin \Delta_{\frR(w)}(K^{k+1})}} \enspace,
		\end{align*}
		where the second equality follows because for a simplex $\sigma \in \Delta_{\frR(w)}(K^{k})$,
		precisely one of the following holds:
		either $\sigma \ni 1$, or $\sigma \not\ni 1$ and $\{1\} \cup \sigma \in \Delta_{\frR(w)}(K^{k+1}) $, or $\sigma \not\ni 1$ and $\{1\} \cup \sigma \notin \Delta_{\frR(w)}(K^{k+1})$.
		Now the statement follows from \cref{thm:betti near cone} and \proofref{thm:n-cycle:i}.\qedhere
\end{proof}

In the proof of \cref{thm:n-cycle:i}, we see that all entries of the $\frR(w)$ are nonzero.
The argument from~\cite[289\psq]{BjornerKalai:1988} than yields an alternative proof of \cref{thm:n-cycle:ii}.
Our final example shows that the assumptions of \cref{thm:n-cycle} are tight.

\begin{example}\label{exmp:tight}
  We return to the triangulation $\RR\PP^2_6$ of the real projective plane from \cref{exmp:rp26:contracted} and  \cref{tab:shifted partial shifts of RP2}.
  Using \OSCAR\ we computed the partial shifts in characteristic zero for all $6! = 720$ permutations on six letters.
  Out of these $120$ are at least as large as the $6$-cycle in the right weak order of $\SymmetricGroup{6}$.
  None of the remaining $600$ partial shifts preserves the rational Betti numbers,
  except for the trivial partial shift by the identity.
  Further, the partial shift $\Delta_{\frR(c_n)}^{\GF{2}}(\RR\PP^2_6)$ has $\GF{2}$-Betti numbers $(1,1,1)$, just like $\RR\PP^2_6$, but it is not a shifted complex.
\end{example}
%
%

\section{Outlook}
\label{sec:outlook}
It seems desirable to derive further results relating the topology of a simplicial complex with its partial shifts.
We believe that partial shifting always (weakly) increases the Betti numbers:
\begin{conjecture}
	\label{thm:betti increase}
	For any simplicial complex $K \subseteq 2^{[n]}$, field $\FF$ and $w \in \SymmetricGroup{n}$,
	we have $\beta^\FF_i(\Delta^\FF_{\frR(w)}(K)) \geq \beta^\FF_i(K)$ for every $i \in \NN$,
	where $\beta_i(K)$ denotes the $i$th Betti number of $K$.
\end{conjecture}
If the above conjecture is true, the following second conjecture would be implied.
\begin{conjecture}
	\label{thm:contracted shift graph acyclic}
	For any $n,k,m \in \NN$, the contracted partial shift graph $\ContractedGraph(n,k,m)$ is acyclic.
\end{conjecture}

It makes sense to consider subgraphs of partial shift graphs with respect to subgroups of $\GL(n,\EE)$.
Particularly interesting are reductive algebraic groups such as the symplectic and orthogonal groups (over perfect fields); see, e.g., Borel~\cite{Borel:1991}.
They also admit a Bruhat decomposition, where the symmetric group $\SymmetricGroup{n}$ (which is the Coxeter group of type A$_{n-1}$) is replaced by a different Coxeter group.
\begin{question}
  What can be said about partial shifting with respect to Coxeter groups of other types?
\end{question}

Note that $2$-uniform hypergraphs are the same as simple graphs.
Moreover, the shifted graphs are precisely the \enquote{threshold graphs} studied in the monograph~\cite{threshold} by Mahadev and Peled.
The following question could be considered in the context of~\cite[Problem~7.2]{Kalai02}:
\begin{question}
  Can we describe the partial shift graph of an arbitrary graph?
  More specifically: for which graphs $G$ does the partial shift graph $\ShiftGraph(G)$ have a unique sink?
\end{question}

Algebraic shifting can also be studied in the polynomial ring rather than the exterior algebra.
This is known as \enquote{symmetric shifting}; we refer to Kalai's survey~\cite{Kalai02} for the precise definition.
Some classical results on algebraic shifting work in both settings.
Sometimes, however, there is a difference.
\begin{question}
  To what extent do our results on partial exterior shifting also hold for symmetric shifting?
\end{question}
Note that symmetric shifting is known to be related to framework rigidity~\cite[\S2.7]{Kalai02}.
In a personal communication, Francisco Santos asked if something can be said about (symmetric or exterior) partial shifting with respect to matrices that are constrained, e.g., in terms of rigidity.


	\sloppy
	\printbibliography
\end{document}